\documentclass[a4paper,11pt,fullpage]{article}
\usepackage[english]{babel}

\usepackage[left=1in,right=1in,top=1in,bottom=1in]{geometry}

\usepackage{amsmath}
\usepackage{amsthm}
\usepackage{amsfonts}
\usepackage{amssymb}
\usepackage{amsxtra}
\usepackage{latexsym}
\usepackage{ifthen}
\usepackage{cite}
\usepackage{esint}
\usepackage[cp1250]{inputenc}

\usepackage{epic}
\usepackage{eepic}

\usepackage[normalem]{ulem}
\usepackage{xcolor}

\DeclareMathOperator*{\osc}{osc}

\numberwithin{equation}{section}
\newtheorem{theorem}{Theorem}[section]
\newtheorem{lemma}{Lemma}[section]
\newtheorem{remark}{Remark}[section]

\newtheorem{proposition}{Proposition}[section]

\def\XXint#1#2#3{{\setbox0=\hbox{$#1{#2#3}{\int}$}
     \vcenter{\hbox{$#2#3$}}\kern-.5\wd0}}

\begin{document}

\title{On the weak Harnack inequality for unbounded non-negative super-solutions of degenerate double-phase parabolic equations
}

\author{ Mariia O. Savchenko, Igor I. Skrypnik, Yevgeniia A. Yevgenieva
 }

  \maketitle

  \begin{abstract} We give a  proof of the weak Harnack inequality for   non-negative super-solutions of degenerate double-phase parabolic equations  under the additional assumption that $u\in L^{s}_{loc}(\Omega_{T})$ with some sufficiently large $s$.




\end{abstract}

\pagestyle{myheadings} \thispagestyle{plain}
\markboth{Mariia O. Savchenko, Igor I. Skrypnik, Yevgeniia A.Yevgenieva}
{On the weak Harnack inequality....}

\section{Introduction}

In this paper, we are concerned with double-phase parabolic equations.
Let $\Omega$ be a domain in $\mathbb{R}^{n}$, $T>0$, $\Omega_{T}:= \Omega \times (0, T)$,
we study unbounded super-solutions to the equation
\begin{equation}\label{eq1.1}
u_{t}-\textrm{div}\mathbb{A}(x, t, \nabla u)=0, \quad (x, t)\in \Omega_{T}.
\end{equation}
Throughout the paper we suppose that the functions $\mathbb{A}:\Omega_{T}\times \mathbb{R}^{n} \rightarrow \mathbb{R}^{n}$ are such that
$\mathbb{A}(\cdot, \cdot, \xi)$ are Lebesgue measurable for all $\xi \in \mathbb{R}^{n}$,
and $\mathbb{A}(x, t, \cdot)$ are continuous for almost all $(x, t)\in \Omega_{T}$.
We also assume that the following structure conditions are satisfied
\begin{equation}\label{eq1.2}
\begin{aligned}
\mathbb{A}(x, t, \xi)\xi &\geqslant  K_{1}\,\big( |\xi|^{p} + a(x,t)|\xi|^{q}\big):=K_{1}\,\varphi(x,t, |\xi|), \quad 2< p < q,
   \\
\ |\mathbb{A}(x, t, \xi)| &\leqslant  K_{2}\big( |\xi|^{p-1} + a(x,t)|\xi|^{q-1}\big)= K_{2}\,\frac{\varphi(x,t,|\xi|)}{|\xi|},
\end{aligned}
\end{equation}
where $K_{1}$, $K_{2}$ are positive constants and function $a(x, t) \geqslant 0 : \Omega_{T}\times \mathbb{R}_{+}\rightarrow \mathbb{R}_{+}$
satisfies the following condition:
\begin{itemize}
\item[($A$)] for any cylinder \\$Q_{r,r^{2}}(x_{0}, t_{0}) := B_{r}(x_{0},t_{0}) \times (t_{0}, t_{0}+r^{2}) \subset Q_{8r,(16r)^{2}}(x_{0}, t_{0}) \subset \Omega_{T}$ there holds
\begin{equation*}
\osc\limits_{Q_{r,r^{2}}(x_{0},t_{0})} a(x,t)\leqslant A r^{\alpha},
\end{equation*}
with some $A>0$ and some $\alpha \in(0,1]$.
\end{itemize}

It is known that for integrands with ($p,q$)-growth, it is crucial that the gap between $p$ and $q$ is not too large. Otherwise, in the case
$q> \dfrac{np}{n-p}$, $p<n$ there exist unbounded minimizers (we refer the reader to \cite{Alk, AlkSur, ArrHue, BarColMin1, BarColMin2, BarColMin3, BurSkr, ColMin1, ColMin2, ColMin3, CupMarMas, HadSkrVoi, HarKinLuk, HarHasLee, HarKuuLukMarPar, RagTac, SavSkrYev2, ShaSkrVoi, Skr, SkrVoi, SkrVoi1, Sur, Wan} for results, references, historical notes and extensive survey of regularity issues). It was Ok \cite{Ok}, who proved the boundedness of minimizers of elliptic functionals of double-phase type  under some additional assumption. More precisely, under the condition $\osc\limits_{B_{r}(x_{0})}a(x)\leqslant A r^{\alpha}$, the minimizer is bounded by a constant depending on $||u||_{L^{s}}$, provided that $\alpha \geqslant q-p$ and $s \geqslant \dfrac{(q-p) n}{\alpha +p -q}$. This condition, for example, gives a possibility to improve the regularity results \cite{BarColMin1, BarColMin2, BarColMin3, ColMin1, ColMin2} for unbounded minimizers with constant depending on $||u||_{L^{s}}$. The weak Harnack inequality for unbounded super-solutions of the corresponding elliptic equations with generalized Orlicz growth under similar condition was proved in \cite{BenHarHasKar}. This result was generalized in \cite{SavSkrYev1}  for unbounded functions from the corresponding De Giorgi's classes $DG^{-}_{\varphi}(\Omega)$.

 The parabolic theory for quasi-linear parabolic equations differs substantially from the elliptic case. 
 This becomes clear by looking at the Barenblatt solution of the parabolic $p$-Laplace equation. DiBenedetto developed an innovative intrinsic scaling method (see \cite{DiB} and the references to
the original papers therein) and proved the H\"{o}lder continuity of weak solutions to \eqref{eq1.1} for $p=q \neq 2$. The intrinsic Harnack inequality for parabolic $p$-Laplace evolution equation was proved in the breakthrough papers \cite{DiBGiaVes1, DiBGiaVes2}. The weak Harnack inequality for parabolic $p$-Laplacian was obtained by Kuusi \cite{Kuu} by using the Krylov-Safonov covering argument. A similar result was proved in \cite{DiBGiaVes3} by using the local clustering lemma. As for parabolic equations with nonstandard growth, this question remains open.

The local boundedness of solutions of parabolic equations is known under the condition $q \leqslant p\dfrac{n+2}{n}$ (see, for example, \cite{Sin}). The upper bound for the number $q$ stems from the parabolic embedding. The intrinsic Harnack inequality for bounded solutions to corresponding singular parabolic equations with ($p,q$)- growth was proved in \cite{Skr}. The weak Harnack inequality for bounded super-solutions of the $p(x)$- Laplace evolution equation was obtained in \cite{Sur}.  In this paper, using DiBenedetto's approach we prove the weak Harnack inequality for unbounded non-negative super-solutions to equation \eqref{eq1.1}  under the condition similar to that of \cite{Ok}. We will focus only on the case $p>2$ and leave the case $p <2< q$  for further research.

To formulate our results, let us remind the reader of the definition of a weak super-solution to equation \eqref{eq1.1}.  We say that function $u$ is a  weak  super-solution to Eq. \eqref{eq1.1} if $u \in V^{2,q}(\Omega_{T}):= C_{\textrm{loc}}(0, T; L^{2}_{\textrm{loc}}(\Omega))\cap L_{\textrm{loc}}^{q}(0, T; W_{\textrm{loc}}^{1,q}(\Omega)),$
and for any compact set $E \subset \Omega$
and every subinterval $[t_{1}, t_{2}]\subset (0, T]$ there holds
\begin{equation}\label{eq1.3}
\int\limits_{E}u \zeta\, dx \bigg|^{t_{2}}_{t_{1}} + \int\limits^{t_{2}}\limits_{t_{1}}\int\limits_{E}
\{-u\zeta_{\tau}+ \mathbb{A}(x, \tau, \nabla u) \nabla \zeta\}\, dx\, d\tau \geqslant\, 0
\end{equation}
 for any testing functions $\zeta \in W^{1,2}(0, T; L^{2}(E))\cap L^{q}(0, T; W_{0}^{1,q}(E))$, $\zeta \geqslant 0$.

Technically, it would be convenient to have a formulation of a weak super-solution that involves $u_{t}$.
Let $\rho(x)\in C_{0}^{\infty}(\mathbb{R}^{n})$, $\rho(x)\geqslant 0$ in $\mathbb{R}^{n}$, $\rho(x)\equiv 0$ for
$|x| > 1$ and $\int\limits_{\mathbb{R}^{n}}\rho (x)\,dx=1$, and set
$$
\rho_{h}(x):= h^{-n}\rho\left(x/h\right), \quad
u_{h}(x, t):= h^{-1}\int\limits_{t}\limits^{t+h}\int\limits_{\mathbb{R}^{n}}u(y, \tau)\rho_{h}(x-y)\,dy\, d\tau .
$$

Fix $t \in (0, T)$ and let $h>0$ be so small that $0<t<t+h<T$.
In \eqref{eq1.3} take $t_{1}=t$, $t_{2}=t+h$ and replace $\zeta$ by $\int\limits_{\mathbb{R}^{n}}\zeta(y, t)\rho_{h}(x-y)\,dy$.
Dividing by $h$, since the testing function does not depend on $\tau$, we obtain
\begin{equation}\label{eq1.4}
\int\limits_{E\times \{t\}} \left(\frac{\partial u_{h}}{\partial t}\, \zeta+[\mathbb{A}(x, t, \nabla u)]_{h} \nabla \zeta\right)dx
\geqslant\,0,
\end{equation}
for all $t \in (0, T-h)$ and for all  $\zeta \in W^{1,q}_{0}(E)$, $\zeta \geqslant 0$.

Our main result reads as follows
\begin{theorem}\label{th1.1}
Let $u$ be a weak super-solution to equation \eqref{eq1.1}, let conditions \eqref{eq1.2} and ($A$) be fulfilled. Assume additionally that
$u \in L^{s}(\Omega_{T})$ and
\begin{equation}\label{eq1.5}
s \geqslant p-2 +\frac{(q-p)(n+p)}{\alpha+p-q}.
\end{equation}
Then there exist positive constants $C_{1}$, $C_{2}$, $C_{3} >0$ depending only on $n$, $p$, $q$, $K_{1}$, $K_{2}$, $A$ and $d:=\big(\iint\limits_{\Omega_{T}} u^{s} dx dt\big)^{\frac{1}{s}}$ such that for a.a. $(x_{0}, t_{0}) \in \Omega_{T}$,
 either
\begin{equation}\label{eq1.6}
\mathcal{I}:= \fint\limits_{B_{\rho}(x_{0})} u(x, t_{0})\, dx \leqslant C_{1}\left\{ \rho + \rho\,\,\psi^{-1}_{Q_{14\rho, (14\rho)^{2}}(x_{0},t_{0})}\bigg(\frac{\rho^{2}}{T-t_{0}}\bigg)\right\},
\end{equation}
or
\begin{equation}\label{eq1.7}
\mathcal{I} \leqslant C_{1} \inf\limits_{B_{4\rho}(x_{0})} u(\cdot, t),
\end{equation}
for all time levels
\begin{equation}\label{eq1.8}
t_{0}+ C_{2}\theta \leqslant t \leqslant t_{0}+ C_{3}\theta,\quad \theta:=\frac{\rho^{2}}
{\psi_{Q_{14\rho, (14\rho)^{2}}(x_{0},t_{0})}(\frac{\mathcal{I}}{\rho})},
\end{equation}
provided that $Q_{16 \rho, (16\rho)^{2}}(x_{0}, t_{0}) \subset \Omega_{T}$. Here \\$\fint\limits_{B_{\rho}(x_{0})} u(x, t_{0})\, dx := |B_{\rho}(x_{0})|^{-1} \int\limits_{B_{\rho}(x_{0})} u(x, t_{0})\, dx$, $\psi_{Q}(v):=\dfrac{\varphi^{+}_{Q}(v)}{v^{2}}= v^{p-2} + a^{+}_{Q} v^{q-2}$, $v>0$,  $a^{+}_{Q}:= \max\limits_{Q} a(x,t)$ and  $\psi^{-1}_{Q}(\cdot)$ is inverse function to $\psi_{Q}(\cdot)$.
\end{theorem}
\begin{remark}
 If inequality \eqref{eq1.6} is violated, i.e.
\begin{equation}\label{eq1.9}
\mathcal{I}  \geqslant C_{1}\left\{ \rho + \rho\,\,\psi^{-1}_{Q_{14\rho, (14\rho)^{2}}(x_{0},t_{0})}\bigg(\frac{\rho^{2}}{T-t_{0}}\bigg)\right\},
\end{equation}
 then  the inclusion
$Q_{16\rho, C_{3}\theta}(x_{0}, t_{0})\subset Q_{16\rho,(16 \rho)^{2}}(x_{0}, t_{0})$ holds,
provided that $C_{1}$ is large enough.  We need this inclusion only in order to use the condition ($A$) in the cylinder $Q_{14\rho,(14\rho)^{2}}(x_{0}, t_{0})$. In the case when $a(x,t)$ does not depend on $t$,
 the first inequality in \eqref{eq1.6} is not required. In this case, it is enough for the inclusion $Q_{16\rho, C_{3}\theta}(x_{0}, t_{0}) \subset \Omega_{T}$, which holds by the second inequality in \eqref{eq1.9}.
\end{remark}

The rest of the paper contains the proof of Theorem~\ref{th1.1}. In Section~\ref{Sec2} we collect some auxiliary propositions and required integral estimates of super-solutions.

The proof of Theorem \ref{th1.1} consists of several steps. First, in Section~\ref{Sec3} we establish the theorem concerning the expansion of positivity (see Proropsition \ref{pr3.1} and Theorem \ref{th3.1}). 
Roughly speaking, it claims that information on the measure of the  positivity set  of $u$ at the time level $\bar{t}$ over the ball $B_{r}(\bar{x})$:
\begin{equation*}
|\{ B_{r}(\bar{x}) : u(\cdot, \bar{t}) \geqslant k \}| \geqslant \beta|B_{r}(\bar{x})|,
\end{equation*}
with some $k >0$,  $\beta \in (0, 1)$ and some  $r>0$: $Q_{6r,(6r)^{2}}(\bar{x}, \bar{t})\subset Q_{\rho,\rho^{2}}(x_{0}, t_{0})$,
translates into an expansion of the positivity set both in space (from a ball $B_{r}(\bar{x})$ to $B_{2r}(\bar{x})$), and in time (from $\bar{t}$ to $t:t>\bar{t}$). 

To prove this statement, we adapt the concepts of DiBenedetto, Gianazza, and Vespri \cite[Chapter 4]{DiBGiaVes3} to the non-standard case. First, we ''expand'' the positivity at some time level $\bar{t}$ on the ball $B_{r}(\bar{x})$ to further time level $t>\bar{t}$ within the same ball (see Lemma~\ref{lem3.1}). Then we ''expand'' it spatially to the larger ball $B_{2r}(\bar{x})$ (Proposition~\ref{pr3.1}). While the proof of the Lemma 3.1 is unified and can be executed without any challenges, the proof of Proposition 3.1 requires more complex approach. As we want to apply the idea of \cite[Chapter 4]{DiBGiaVes3} and obtain pointwise estimates for the newly defined unknown functions (different, depending on the phase, see \eqref{eq1.10}, \eqref{eq1.11} below), we are faced with technical difficulties. To overcome them we divide the proof into two cases: the first is the so-called case of the $p\,$-phase
\begin{equation}\label{eq1.10}
a^{+}_{Q_{6\rho,(6\rho)^{2}}(x_{0},t_{0})}\,\bigg(\frac{\bar{\varepsilon} k}{\rho}\bigg)^{q-p} \leqslant 1,
\end{equation}
and the second is the case of the so-called ($p, q$)-phase 
\begin{equation}\label{eq1.11}
a^{+}_{Q_{6\rho,(6\rho)^{2}}(x_{0},t_{0})}\,\bigg(\frac{\bar{\varepsilon} k}{\rho}\bigg)^{q-p} \geqslant 1,
\end{equation}
with some sufficiently small $\bar{\varepsilon} \in (0,1)$. The definition of $p$-phase and ($p, q$)-phase is technical in our case, so it differs from the corresponding ones in the elliptic case (see, for example \cite{BarColMin1}). Moreover, we need to split the case of $p\,$-phase into two subcases, namely, $p\,$- and ($p, q$)-phase for the smaller cylinder $Q_{6r,(6r)^{2}}(\bar{x}, \bar{t})$ (see cases $(i)$ and $(ii)$ in Section 3). It seems that the most difficult case is when the $p$-phase condition holds in the large cylinder $Q_{6\rho,(6\rho)^{2}}(x_{0},t_{0})$, and simultaneously the ($p, q$)-phase condition holds in a small cylinder $Q_{6r,(6r)^{2}}(\bar{x}, \bar{t})$. For all of the listed cases, the proof of Proposition~\ref{pr3.1} is provided separately in subsections \ref{subsec.3.1}, \ref{subsec.3.2}, \ref{subsec.3.3}. 

Finally, the main "expansion of positivity" Theorem~\ref{th3.1} is formulated and proved in subsection~\ref{subsec.3.4}. In this theorem we require additional technical assumption on the number $k$:
\begin{equation}\label{eq1.12}
k^{s}\leqslant \gamma_{0}\bigg(\frac{k^{p-2}}{\rho^{n+p}}+ a^{+}_{Q_{6\rho,(6\rho)^{2}}(x_{0}, t_{0})}\frac{k^{q-2}}{\rho^{n+q}}\bigg),
\end{equation}
with some $\gamma_{0}>1$.

The proof of Theorem~\ref{th1.1} is outlined in Section~\ref{Sec4}. First, to be able to apply "expansion of positivity" Theorem~\ref{th3.1}, we need to check assumption \eqref{eq1.12} for $k= \mathcal{I}$. This is done in Lemma \ref{lem4.1}. We also note that in the elliptic case,  inequality \eqref{eq1.12} translates into $k^{s}\leqslant \dfrac{\gamma_{0}}{\rho^{n}}$ and gives us the possibility to use 
 condition ($A 1-s_{*}$) from \cite{BenHarHasKar}, i.e. we can obtain an estimate
\begin{equation*}
\rho^{\alpha+p-q} k^{q-p}\leqslant \gamma_{0}^{\frac{q-p}{s}}\rho^{\alpha+p-q-(q-p)\frac{n}{s}}\leqslant  \gamma_{0}^{\frac{q-p}{s}},
\end{equation*}
provided that $s\geqslant \dfrac{(q-p)n}{\alpha+p-q}$. Inequality \eqref{eq1.12} plays a similar role  in our case. And while verifying inequality \eqref{eq1.12} in the elliptic case is straightforward (see \cite{SavSkrYev1}), in our scenario, some effort will be required.

After we have proved the theorem on the expansion of positivity, we prove the weak Harnack inequality. Following Kuusi \cite{Kuu}  we distinguish two alternative cases : ''hot'' and ''cold''. The ''hot'' case means that super-solution attains large values compared with $\mathcal{I}$ in a relatively large set. To prove the Theorem~\ref{th1.1}, we apply the expansion of positivity theorem but do not use the classical covering argument of Krylov and Safonov \cite{KrySaf}, DiBenedetto and  Trudinger \cite{DiBTru} as it was done in \cite{Kuu}.  Instead, we use the  local clustering lemma due to DiBenedetto,Gianazza and Vespri \cite{DiBGiaVes3}. The proof is provided in subsection~\ref{subsect4.1}.

The ''cold'' case means that the spatial integrals of super-solutions to a small power are estimated by $\mathcal{I}$ at all time levels. By applying  Moser's iterative technique, this case leads to a uniform estimate of  super-solutions to a sufficiently large power.  Finally, we use the expansion of positivity theorem again to complete the proof of the weak Harnack inequality. All the arguments of the proof in ''cold'' case are introduced in subsection~\ref{subsect4.2}

Another challenge encountered in establishing our main result is related to the unboundedness of super-solutions. Consequently, we resort to using averages of $u$ over the cylinders instead of $\sup u$. Additionally, we need to obtain estimates of $\mathcal{I}$ by $\|u\|_{L^s}$.

\section{Auxiliary material and integral estimates }\label{Sec2}

\subsection{ Local clustering Lemma}\label{subsect2.1}

The following  lemma will be used in the sequel, it is  the local clustering lemma, see \cite{DiBGiaVes3}.

\begin{lemma}\label{lem2.1}
Let $K_{r}(y)$ be a cube in $\mathbb{R}^{n}$ of edge $r$ centered at $y$ and let $u\in W^{1,1}(K_{r}(y))$ satisfy
\begin{equation}\label{eq2.1}
||(u-k)_{-}||_{W^{1,1}(K_{r}(y))} \leqslant \mathcal{K}\,k\,r^{n-1},\,\,\,\,\,\,and\,\,\,\,\,\, |\{K_{r}(y) : u\geqslant k \}|\geqslant \beta |K_{r}(y)|,
\end{equation}
with some $\beta \in (0,1)$, $k\in\mathbb{R}^{1}$ and $\mathcal{K} >0$. Then for any $\xi \in (0,1)$ and any $\nu\in (0,1)$ there exists $\bar{x} \in K_{r}(y)$ and $\delta=\delta(n) \in (0,1)$ such that
\begin{equation}\label{eq2.2}
|\{K_{\bar{r}}(\bar{x}):  u\geqslant \xi\,k \}| \geqslant (1-\nu) |K_{\bar{r}}(y)|,\,\,\, \bar{r}:=\delta \beta^{2}\frac{(1-\xi)\nu}{\mathcal{K}}\,r.
\end{equation}
\end{lemma}

\subsection{De Giorgi--Poincare lemma}\label{subsect2.2}
The following lemma is the well-known De Giorgi--Poincare lemma (see for example~\cite{DiBGiaVes3}).
\begin{lemma}\label{lem2.2}
Let $u\in W^{1,1}(B_{r}(y))$ for some $r>0$. Let $k$ and $l$ be real numbers such that
$k < l$. Then there exists a constant $\gamma >0$ depending only on $n$ such that
\begin{equation*}
(l-k) |A_{k,r}| |B_{r}(y)\setminus A_{l,r}| \leqslant \gamma r^{n+1} \int\limits_{A_{l,r}\setminus A_{k,r}}|\nabla u|\, dx,
\end{equation*}
where $A_{k,r} = \{x \in B_{r}(y) : u(x) < k\}.$
\end{lemma}

\subsection{Local energy estimates}\label{subsec2.3}

We refer to the parameters $n$, $p$, $q$, $K_{1}$, $K_{2}$, $A$ and $d$ as our structural data, and we write $\gamma$ if it can be quantitatively determined a priori only in terms of the above quantities. The generic constant $\gamma$ may change from line to line.

\begin{lemma}\label{lem2.3}
 Let $u$ be a weak non-negative super-solution to equation \eqref{eq1.1}, then for  any $Q_{r,\eta}(\bar{x},\bar{t})\subset Q_{r,r}(\bar{x},\bar{t})  \subset Q_{8r, 8r}(\bar{x},\bar{t}) \subset \Omega_{T}$, any $k >0$, any $\sigma \in (0,1)$, any  $\zeta_{1}(x) \in C^{\infty}_{0}(B_{r}(\bar{x}))$,  $0 \leqslant \zeta_{1}(x) \leqslant 1$, $\zeta_{1}(x)=1$ in $B_{r(1-\sigma)}(\bar{x})$, $|\nabla \zeta_{1}(x)| \leqslant (\sigma r)^{-1}$, any  $\zeta_{2}(t)\in C^{1}(\mathbb{R}_{+} ), 0\leqslant \zeta_{2}(t)\leqslant 1$, $\zeta_{2}(t)=1$ for  $t\leqslant \bar{t}+\eta(1-\sigma)$, $\zeta_{2}(t)=0$ for $t\geqslant \bar{t}+\eta$,  $|\frac{d}{dt}\zeta_{2}(t)|\leqslant (\sigma \eta)^{-1}$  next inequalities hold
\begin{multline}\label{eq2.3}
\sup\limits_{\bar{t}<t<\bar{t}+ \eta}\int\limits_{B_{r}(\bar{x})}(u-k)^{2}_{-}\big(\zeta_{1}\zeta_{2}\big)^{q}dx +
\\+\gamma^{-1}\bigg(1+a^{-}_{Q_{r,\eta}(\bar{x},\bar{t})}\bigg(\frac{k}{r}\bigg)^{q-p}\bigg)\iint\limits_{Q^{+}_{r,\eta}(\bar{x}, \bar{t})}|\nabla (u-k)_{-}|^{p}\big(\zeta_{1}\zeta_{2}\big)^{q} dx\,dt \leqslant \\ \leqslant
\gamma \sigma^{-q} \varphi^{+}_{Q_{r,\eta}(\bar{x},\bar{t})}\left(\frac{k}{r}\right)\, \bigg\{1+\frac{k^{2}}{\eta \varphi^{+}_{Q_{r,\eta}(\bar{x},\bar{t})}(\frac{k}{r})}\bigg\} |A^{-}_{k,r,\eta}|,
\end{multline}
\begin{equation}\label{eq2.4}
\sup\limits_{\bar{t}<t<\bar{t}+ \eta}\int\limits_{B_{r}(\bar{x})}(u-k)^{2}_{-}\zeta_{1}^{q}dx \leqslant
\int\limits_{B_{r}(\bar{x})\times\{\bar{t}\}}(u-k)^{2}_{-}\zeta_{1}^{q}dx+ \gamma\,\sigma^{-q}\, \varphi^{+}_{Q_{r,\eta}(\bar{x},\bar{t})}\left(\frac{k}{r}\right)\, |A^{-}_{k,r,\eta}|,
\end{equation}
where $A^{-}_{k,r,\eta}= Q_{r,\eta}(\bar{x},\bar{t})\cap \big\{u \leqslant k\big\}$,  $\varphi^{+}_{Q_{r,\eta}(\bar{x},\bar{t})}\big(\frac{k}{r}\big)=\big(\frac{k}{r}\big)^{p}+ a^{+}_{Q_{r,\eta}(\bar{x},\bar{t})}\big(\frac{k}{r}\big)^{q}$, \\$a^{-}_{Q_{r,\eta}(\bar{x},\bar{t})}=\min\limits_{Q_{r,\eta}(\bar{x},\bar{t})}a(x,t)$, $a^{+}_{Q_{r,\eta}(\bar{x},\bar{t})}=\max\limits_{Q_{r,\eta}(\bar{x},\bar{t})}a(x,t)$.
\end{lemma}

\begin{proof}
Test \eqref{eq1.4} by $(u_{h}-k)_{-}\big(\zeta_{1}\zeta_{2}\big)^{q}$, integrating over $(\bar{t}, \bar{t}+\eta)$,
letting $h\rightarrow 0$, using conditions \eqref{eq1.2}, ($A$)  and the Young inequality we arrive at
\begin{multline*}
\sup\limits_{\bar{t}<t<\bar{t}+ \eta}\int\limits_{B_{r}(\bar{x})}(u-k)^{2}_{-}\big(\zeta_{1}\zeta_{2}\big)^{q}dx +\gamma^{-1}\iint\limits_{Q_{r,\eta}(\bar{x}, \bar{t})}|\nabla (u-k)_{-}|^{p}\big(\zeta_{1}\zeta_{2}\big)^{q} dx\,dt +  \\+\gamma^{-1}\iint\limits_{Q_{r,\eta}(\bar{x}, \bar{t})}a(x,t) |\nabla (u-k)_{-}|^{q}\big(\zeta_{1}\zeta_{2}\big)^{q} dx\,dt  \leqslant \\ \leqslant
\gamma \sigma^{-1} \frac{k^{2}}{\eta} |A^{-}_{k,r,\eta}|+\gamma \sigma^{-q}\iint\limits_{A^{-}_{k,r,\eta}}\varphi\left(x,t,\frac{k}{r}\right) dx\,dt \leqslant \\ \leqslant \gamma \sigma^{-q}\left(\frac{k^{2}}{\eta}+ \varphi^{+}_{Q_{r,\eta}(\bar{x},\bar{\eta})}\left(\frac{k}{r}\right)\right)|A^{-}_{k,r,\eta}|.
\end{multline*}
Using the Young inequality we obtain
\begin{multline*}
\bigg(1+a^{-}_{Q_{r,\eta}(\bar{x},\bar{t})}\bigg(\frac{k}{r}\bigg)^{q-p}\bigg)\iint\limits_{Q_{r,\eta}(\bar{x}, \bar{t})}|\nabla (u-k)_{-}|^{p}\big(\zeta_{1}\zeta_{2}\big)^{q} dx\,dt \leqslant 
\\ \leqslant\iint\limits_{Q_{r,\eta}(\bar{x}, \bar{t})}|\nabla (u-k)_{-}|^{p}\big(\zeta_{1}\zeta_{2}\big)^{q} dx\,dt+\\+
\iint\limits_{Q_{r,\eta}(\bar{x}, \bar{t})}a(x,t)|\nabla (u-k)_{-}|^{q}\big(\zeta_{1}\zeta_{2}\big)^{q} dx\,dt+
a^{+}_{Q_{r,\eta}(\bar{x},\bar{t})}\bigg(\frac{k}{r}\bigg)^{q}|A^{-}_{k,r,\eta}|,
\end{multline*}
from which the required inequality \eqref{eq2.3} follows. Now test \eqref{eq1.4} by $(u_{h}-k)_{-}\zeta_{1}^{q}$,
completely similar to the previous we arrive at \eqref{eq2.4}, this proves the lemma.
\end{proof}

\begin{lemma}\label{lem2.4}
Let $u$ be a weak non-negative super-solution to equation \eqref{eq1.1}, then for  any $Q_{r,\eta}(\bar{x},\bar{t})\subset Q_{r,r^{2}}(\bar{x},\bar{t})  \subset Q_{8r,(8r)^{2}}(\bar{x}, \bar{t}) \subset \Omega_{T}$, any $\delta>0$, any $\varepsilon, \sigma \in (0,1)$,    next inequality holds
\begin{multline}\label{eq2.5}
\frac{1}{1-\varepsilon}\sup\limits_{\bar{t}<t<\bar{t}+ \eta}\int\limits_{B_{r}(\bar{x})}(u+\delta)^{1-\varepsilon}\big(\zeta_{1}\zeta_{2}\big)^{q}dx +\frac{\varepsilon}{\gamma}\iint\limits_{Q_{r,\eta}(\bar{x}, \bar{t})}(u+\delta)^{-\varepsilon -1}|\nabla u|^{p}\big(\zeta_{1}\zeta_{2}\big)^{q} dx\,dt + \\+ \frac{\varepsilon}{\gamma}\iint\limits_{Q_{r,\eta}(\bar{x}, \bar{t})}a(x,t)(u+\delta)^{-\varepsilon -1} |\nabla u|^{q}\big(\zeta_{1}\zeta_{2}\big)^{q} dx\,dt \leqslant \frac{1}{(1-\varepsilon)\sigma \eta}\iint\limits_{Q_{r,\eta}(\bar{x}, \bar{t})}(u+\delta)^{1-\varepsilon}dx\,dt +\\ +\frac{\gamma \varepsilon^{1-p}}{(\sigma r)^{p}}\iint\limits_{Q_{r,\eta}(\bar{x}, \bar{t})}(u+\delta)^{p-\varepsilon -1}dx\,dt + \frac{\gamma \varepsilon^{1-q}}{(\sigma r)^{q}}a^{+}_{Q_{r,\eta}(\bar{x},\bar{t})}\iint\limits_{Q_{r,\eta}(\bar{x}, \bar{t})}(u+\delta)^{q-\varepsilon -1} dx\,dt.
\end{multline}
\end{lemma}
\begin{proof}
Test \eqref{eq1.4} by $(u_{h}+\delta)^{-\varepsilon}\big(\zeta_{1}\zeta_{2}\big)^{q}$, integrating over $(\bar{t}, \bar{t}+\eta)$,
letting $h\rightarrow 0$, using conditions \eqref{eq1.2}  and the Young inequality we arrive at the required \eqref{eq2.5}.
\end{proof}

\subsection{De Giorgi type lemma}\label{subsec2.4}
Let $u\in V^{2,m}(\Omega_{T})$, $m>\frac{2n}{n+1}$, $u\geqslant 0$ and let next inequalities hold
\begin{multline}\label{eq2.6}
\sup\limits_{\bar{t}<t<\bar{t}+ \eta}\int\limits_{B_{r}(\bar{x})}(u-k)^{2}_{-}\big(\zeta_{1}\zeta_{2}\big)^{m}dx +
\gamma^{-1}\iint\limits_{Q_{r,\eta}(\bar{x}, \bar{t})}|\nabla (u-k)_{-}|^{m}\big(\zeta_{1}\zeta_{2}\big)^{m} dx\,dt \leqslant\\
\leqslant K \sigma^{-m} \left\{\frac{k^{2}}{\eta} + \bigg(\frac{k}{r}\bigg)^{m} \right\}|A^{-}_{k,r,\eta}|,
\end{multline}
for any $k>0$, any cylinder $Q_{r,\eta}(\bar{x},\bar{t}) \subset Q_{8r,8\eta}(\bar{x},\bar{t})\subset \Omega_{T}$ and any $\sigma\in(0,1)$, and with some $K>0$.
Here $\zeta_{1}$, $\zeta_{2}$ and $A^{-}_{k,r,\eta}$ were defined in Lemma \ref{lem2.3}.

The following lemma is the standard De Giorgi-type lemma (cf. \cite{DiBGiaVes3}, Chapter~3).
\begin{lemma}\label{lem2.6}
Let \eqref{eq2.6} hold, then there exists $\nu\in(0,1)$ depending only on $K$, $n$, $m$ and  $r$, $\eta$ such that if
\begin{equation}\label{eq2.7}
\big|\big\{(x,t)\in Q_{r,\eta}(\bar{x}, \bar{t}) : u(x,t) \leqslant k \big\}\big| \leqslant \nu |Q_{r,\eta}(\bar{x}, \bar{t})|,
\end{equation}
then
\begin{equation}\label{eq2.8}
u(x,t) \geqslant \frac{k}{2},\quad (x,t)\in  Q_{\frac{r}{2},\frac{\eta}{2}}(\bar{x}, \bar{t}).
\end{equation}
\end{lemma}
The number $\nu$ is chosen to satisfy
\begin{equation}\label{eq2.9}
\nu:=\frac{1}{\gamma} \frac{r^{m}}{\eta k^{m-2}} \bigg(1+ \frac{\eta k^{m-2}}{r^{m}}\bigg)^{-\frac{n+m}{m}}.
\end{equation}

\section{Expansion of positivity}\label{Sec3}
As it was already mentioned,   we follow \cite{DiBGiaVes3}, using the idea of introducing new unknown functions that depend on the phase.

Fix $(x_{0}, t_{0}) \in \Omega_{T}$ such that $Q_{8\rho, (8\rho)^{2}}(x_{0}, t_{0}) \subset \Omega_{T}$ and let $Q_{8r,(8r)^{2}}(\bar{x},\bar{t})\subset Q_{\rho, \rho^{2}}(x_{0}, t_{0})$.
In what follows, we will assume that  $k>0$ satisfies conditions
\begin{equation}\label{eq3.1}
C_{*} \rho\leqslant k,\quad k^{s} \leqslant \gamma_{0} \frac{\varphi^{+}_{Q_{6\rho,(6\rho)^{2}}(x_{0},t_{0})}(\frac{k}{\rho})}{\rho^{n} k^{2}}= \gamma_{0} \bigg(\frac{k^{p-2}}{\rho^{n+p}}+a^{+}_{Q_{6\rho,(6\rho)^{2}}(x_{0},t_{0})}\frac{k^{q-2}}{\rho^{n+q}}\bigg),
\end{equation}
where $\gamma_{0}>1$ is some fixed  number depending only on the data and $C_{*} >1$ depends only on the known data to be chosen later.

First, we will prove the following result.
\begin{proposition}\label{pr3.1}
Let $u$ be a weak non-negative super-solution to equation \eqref{eq1.1}, let $k$ satisfy  \eqref{eq3.1}
and let also
\begin{equation}\label{eq3.2}
\big|\big\{B_{r}(\bar{x}): u(\cdot, \bar{t}) > k \big\}\big|\geqslant \beta_{0} |B_{r}(\bar{x})|,
\end{equation}
with some $\beta_{0} \in (0,1)$. Then there exist numbers $C_{*}$, $b_{1}$, $b_{2}>0$ and $\varepsilon_{0}$, $\sigma_{0}\in (0,1)$ depending only on the data and $\beta_{0}$ such that
\begin{equation}\label{eq3.3}
u(x,t) \geqslant \sigma_{0} k, \quad x\in B_{2r}(\bar{x}),
\end{equation}
for all
\begin{equation}\label{eq3.4}
\bar{t}+\eta_{1}:=\bar{t}+b_{1}\frac{(\sigma_{0}k)^{2}}{\varphi^{+}_{Q_{6r,(6r)^{2}}(\bar{x},\bar{t})}\big(\frac{\sigma_{0}k}{r}\big)}\leqslant t \leqslant \bar{t}+
b_{2}\frac{(\sigma_{0}k)^{2}}{\varphi^{+}_{Q_{6r,(6r)^{2}}(\bar{x},\bar{t})}\big(\frac{\sigma_{0}k}{r}\big)}:=\bar{t}+\eta_{2}.
\end{equation}
\end{proposition}

\begin{lemma}\label{lem3.1}
Let conditions of Proposition \ref{pr3.1} hold, then there exist $\varepsilon, \delta \in(0,1)$, depending only on the data and $\beta_{0}$ such that 
\begin{equation}\label{eq3.5}
\big|\big\{B_{r}(\bar{x}) : u(\cdot, t) \geqslant \varepsilon e^{-\tau} k\big\}\big| \geqslant \frac{\beta_{0}}{4} |B_{r}(\bar{x})|,
\end{equation}
for any $0 \leqslant \tau\leqslant \log C_{*}$ and for any $\bar{t} <t \leqslant \bar{t}+ \dfrac{\delta \big(k e^{-\tau}\big)^{2}}{\varphi^{+}_{Q_{6r,(6r)^{2}}(\bar{x},\bar{t})}(\frac{k e^{-\tau}}{r})}$. Here $C_{*}$ is the constant from inequality \eqref{eq3.1}.
\end{lemma}
\begin{proof}
Use inequality \eqref{eq2.4} with $k$ replaced by $k e^{-\tau}$, $\eta= \dfrac{\delta \big(k e^{-\tau}\big)^{2}}{\varphi^{+}_{Q_{6r,(6r)^{2}}(\bar{x},\bar{t})}(\frac{k e^{-\tau}}{r})}$. We need  to check the inequality  
$$\varphi^{+}_{Q_{r,\eta}(\bar{x},\bar{t})}\Big(\dfrac{k e^{-\tau}}{r}\Big)\leqslant \varphi^{+}_{Q_{r,r^{2}}(\bar{x},\bar{t})}\Big(\dfrac{k e^{-\tau}}{r}\Big),$$ 
which holds if $\eta \leqslant r^{2}$. We have $\eta\leqslant r^{p}(ke^{-\tau})^{2-p}\leqslant r^{2},$ provided that $ke^{-\tau}\geqslant r$.
By the first inequality in \eqref{eq3.1} this inequality is valid if $e^{\tau}\leqslant C_{*}$.

So, by \eqref{eq3.2} we obtain
\begin{multline*}
\sup\limits_{\bar{t}<t<\bar{t}+\eta}\int\limits_{B_{r(1-\sigma)}(\bar{x})}(u-ke^{-\tau})^{2}_{-}\,dx\leqslant
\int\limits_{B_{r}(\bar{x})\times\{\bar{t}\}}(u-ke^{-\tau})^{2}_{-}\,dx+ \gamma \sigma^{-q}\,\eta \varphi^{+}_{Q_{r,\eta}(\bar{x},\bar{t})}\Big(\frac{k e^{-\tau}}{r}\Big)|B_{r}(\bar{x})|\leqslant\\ \leqslant \big(1-\beta_{0}+\gamma \sigma^{-q} \delta \big)\big(k e^{-\tau}\big)^{2}|B_{r}(\bar{x})|,
\end{multline*}
Therefore, for any $\bar{t}<t<\bar{t}+ \eta$
\begin{equation*}
\big|\big\{B_{r}(\bar{x}) : u(\cdot, t) \leqslant \varepsilon k e^{-\tau} \big\}\big| \leqslant n\sigma|B_{r}(\bar{x})| +\frac{1}{(1-\varepsilon)^{2}}\big\{1-\beta_{0}+\gamma \sigma^{-q} \delta \big\} |B_{r}(\bar{x})|,
\end{equation*}
provided that $0<\tau \leqslant \log C_{*}$.

Choose
\begin{equation}\label{eq3.6}
\sigma=\frac{\beta_{0}}{8 n},\quad \frac{\beta_{0}}{8}\leqslant \varepsilon=1-\frac{(1-\frac{3}{4}\beta_{0})^{\frac{1}{2}}}{(1-\frac{1}{2}\beta_{0})^{\frac{1}{2}}}\leqslant \frac{\beta_{0}}{2},\quad \delta=\frac{\beta_{0}^{q+1}}{32 (8n)^{q}\gamma},
\end{equation}
we arrive at the required \eqref{eq3.5}, which proves the lemma.
\end{proof}
In the proof of Proposition \ref{pr3.1} we will distinguish between two different cases. Fix $j_{*} >1$, which will be chosen later depending only on the data and $\beta_{0}$ and set $\tau_{*}= \Big(\dfrac{2^{j_{*}}}{\varepsilon}\Big)^{q-2}$. The first one is the so-called case of $p$-phase
\begin{equation}\label{eq3.7}
a^{+}_{Q_{6\rho,(6\rho)^{2}}(x_{0},t_{0})}\bigg(\frac{\bar{\varepsilon} k}{\rho}\bigg)^{q-p} \leqslant 1,\quad \bar{\varepsilon}=\varepsilon 2^{-j_{*}} e^{-\tau_{*}}
\end{equation}
and the second is the case of ($p,q$)-phase
\begin{equation}\label{eq3.8}
a^{+}_{Q_{6\rho,(6\rho)^{2}}(x_{0},t_{0})}\bigg(\frac{\bar{\varepsilon} k}{\rho}\bigg)^{q-p} \geqslant 1.
\end{equation}
In turn, we divide case \eqref{eq3.7} into subcases.  We will assume that either
\begin{itemize}
\item[($i$)]
\begin{equation*}
a^{+}_{Q_{6r,(6r)^{2}}(\bar{x},\bar{t})}\,\bigg(\frac{\bar{\varepsilon} k}{r}\bigg)^{q-p} \leqslant 1\quad \text{and}\quad a^{+}_{Q_{6\rho,(6\rho)^{2}}(x_{0},t_{0})}\bigg(\frac{\bar{\varepsilon} k}{\rho}\bigg)^{q-p} \leqslant 1,
\end{equation*}
\end{itemize}
or
\begin{itemize}
\item[($ii$)]
\begin{equation*}
a^{+}_{Q_{6r,(6r)^{2}}(\bar{x},\bar{t})}\,\bigg(\frac{\bar{\varepsilon} k}{r}\bigg)^{q-p} \geqslant 1 \quad \text{and}\quad a^{+}_{Q_{6\rho,(6\rho)^{2}}(x_{0},t_{0})}\bigg(\frac{\bar{\varepsilon} k}{\rho}\bigg)^{q-p} \leqslant 1.
\end{equation*}
\end{itemize}

\subsection{Proof of Proposition \ref{pr3.1} in the case  ($i$)}\label{subsec.3.1}
For $\tau \geqslant \bar{\tau}_{*}:=\tau_{*}+ j_{*}\log2 +\log\dfrac{1}{\varepsilon}$ we have
\begin{equation*}
a^{+}_{Q_{6r,(6r)^{2}}(\bar{x},\bar{t})}\,\bigg(\frac{k}{r e^{\tau}}\bigg)^{q-p}\leqslant  a^{+}_{Q_{6r,(6r)^{2}}(\bar{x},\bar{t})}\bigg(\frac{\varepsilon k}{r 2^{j_{*}} e^{\tau_{*}}}\bigg)^{q-p}= a^{+}_{Q_{6r,(6r)^{2}}(\bar{x},\bar{t})}\,\bigg(\frac{\bar{\varepsilon} k}{r}\bigg)^{q-p} \leqslant 1,
\end{equation*}
therefore, $\varphi^{+}_{Q_{6r,(6r)^{2}}(\bar{x},\bar{t})}\bigg(\dfrac{ k}{r e^{\tau}}\bigg) \leqslant 2\bigg(\dfrac{ k}{r e^{\tau}}\bigg)^{p}$ for $\tau \geqslant \bar{\tau}_{*}$ and  inequality \eqref{eq3.5} yields
\begin{equation*}
\left|\left\{B_{r}(\bar{x}) : u\left(\cdot, \bar{t}+\frac{\delta}{2}r^{p}\bigg(\frac{e^{\tau}}{k}\bigg)^{p-2}\right) \geqslant  e^{-\tau}\varepsilon  k\right\}\right| \geqslant \frac{\beta_{0}}{4} |B_{r}(\bar{x})|,\quad \text{for all}\quad \tau \geqslant\bar{\tau}_{*}.
\end{equation*}
Following \cite{DiBGiaVes3} we consider the function
$$w(y,\tau):=\frac{e^{\tau}}{k} u\left(\bar{x} +r y, \bar{t}+ \frac{\delta}{2}r^{p}\bigg(\frac{e^{\tau}}{k}\bigg)^{p-2}\right),\quad \tau \geqslant \bar{\tau}_{*}.$$
The previous inequality  translates into $w$ as $\big|\big\{B_{1}(0) : w(\cdot, \tau) \geqslant \varepsilon \big\}\big| \geqslant \dfrac{\beta_{0}}{4} |B_{1}(0)|$, which yields
\begin{equation}\label{eq3.9}
\big|\big\{B_{4}(0) : w(\cdot, \tau) \geqslant \varepsilon \big\}\big| \geqslant \dfrac{\beta_{0}}{ 4^{n+1}} |B_{4}(0)|, \quad \text{for all}\quad \tau \geqslant \bar{\tau}_{*}.
\end{equation}
Since $w \geqslant 0$, formal differentiation, which can be justified in a standard way, gives
\begin{equation}\label{eq3.10}
w_{\tau}= w + \frac{\delta}{2} (p-2) r^{p}\bigg(\frac{e^{\tau}}{k}\bigg)^{p-1} u_{t} \geqslant div\, \bar{\mathbb{A}}(y,\tau,\nabla w),
\end{equation}
where $\bar{\mathbb{A}}$ satisfies the inequalities
\begin{equation}\label{eq3.11}
\begin{aligned}
\ \bar{\mathbb{A}}(y,\tau,\nabla w) \nabla w \geqslant K_{1} \frac{\delta}{2} (p-2)\left(|\nabla w|^{p} + \bar{a}(y,\tau)\bigg(\frac{k}{e^{\tau}r}\bigg)^{q-p}|\nabla w|^{q}\right),
\\
\ |\bar{\mathbb{A}}(y,\tau,\nabla w)| \leqslant K_{2} \frac{\delta}{2} (p-2)\left(|\nabla w|^{p-1} + \bar{a}(y,\tau)\bigg(\frac{k}{e^{\tau}r}\bigg)^{q-p}|\nabla w|^{q-1}\right),
\end{aligned}
\end{equation}
and $\bar{a}(y,\tau):= a\big(\bar{x} +r y, \bar{t}+ \frac{\delta}{2}r^{p}\big(\frac{e^{\tau}}{k}\big)^{p-2}\big).$

\begin{lemma}\label{lem3.2}
For any $\nu\in (0,1)$ there exists $j_{*}$, depending only on the data, $\beta_{0}$ and $\nu$, such that
\begin{equation}\label{eq3.12}
\left|\left\{Q^{*}: w \leqslant \frac{\varepsilon}{2^{j_{*}}} \right\}\right| \leqslant \nu |Q^{*}|,
\end{equation}
$Q^{*}= B_{4}(0)\times (\bar{\tau}_{*}+ \frac{1}{2}\big(\frac{2^{j_{*}}}{\varepsilon}\big)^{p-2}, \bar{\tau}_{*}+ \big(\frac{2^{j_{*}}}{\varepsilon}\big)^{p-2}).$
\end{lemma}
\begin{proof}
Using Lemma \ref{lem2.2} with $k=k_{j}:=\frac{\varepsilon}{2^{j}}$ and $l= k_{j-1}$, $1\leqslant j \leqslant j_{*}$, due to \eqref{eq3.9} we obtain
\begin{equation*}
k_{j} |A_{k_{j},4}(\tau)| \leqslant \gamma(\beta_{0})  \int\limits_{A_{k_{j-1},4}(\tau)\setminus A_{k_{j},4}(\tau)}|\nabla w| \,dx,\quad \tau \geqslant \bar{\tau}_{*},
\end{equation*}
where $A_{k_{j},4}(\tau):= B_{4}(0)\cap \{w(\cdot, \tau) < k_{j}\}$.

Integrating this inequality with respect to $\tau$,  $\tau\in(\bar{\tau}_{*} + \frac{1}{2}k^{2-p}_{j_{*}}, \bar{\tau}_{*}+  k^{2-p}_{j_{*}})$ and using the H\"{o}lder inequality
we have
\begin{equation}\label{eq3.13}
k^{\frac{p}{p-1}}_{j} |A_{j}|^{\frac{p}{p-1}} \leqslant \gamma(\beta_{0}) \left(\,\,\iint\limits_{A_{j-1}}| \nabla w|^{p} dy\,d\tau\right)^{\frac{1}{p-1}}|A_{j-1}\setminus A_{j}|,
\end{equation}
where $A_{j}:= \int\limits_{\bar{\tau}_{*}+\frac{1}{2}k^{2-p}_{j_{*}}}^{\bar{\tau}_{*}+k^{2-p}_{j_{*}}}A_{k_{j},4}(\tau)\,d\tau$.

To estimate the first factor on the right-hand side of \eqref{eq3.13}, similarly to Lemma \ref{lem2.3} with $|\frac{d}{d\tau}\zeta_{2}|\leqslant \gamma k^{p-2}_{j_{*}}$, by structure inequalities \eqref{eq3.11} we obtain
\begin{multline*}
\sup\limits_{\bar{\tau}_{*}+\frac{1}{2}k^{2-p}_{j_{*}}<\tau <\bar{\tau}_{*}+  k^{2-p}_{j_{*}}}\int\limits_{B_{4}(0)}(w-k_{j-1})_{-}^{2}\,dx+
\iint\limits_{A_{j-1}}| \nabla w|^{p} dy\,d\tau \leqslant \\ \leqslant \gamma\big( k^{2}_{j-1}k^{p-2}_{j_{*}}+ k^{p}_{j-1}\big)|Q_{1}^{*}\cap\{w<k_{j-1}\}|+ \gamma k^{q}_{j-1}\iint\limits_{Q_{1}^{*}\cap\{w<k_{j-1}\}}\bar{a}(y,\tau)\bigg(\frac{k}{e^{\tau}r}\bigg)^{q-p}\,dy\, d\tau
\leqslant \\ \leqslant \gamma k^{p}_{j} |Q^{*}\cap\{w<k_{j-1}\}|\left\{1+  \,\bigg(\frac{k_{j}k}{e^{\bar{\tau}_{*}}r}\bigg)^{q-p}\,\bar{a}^{+}_{Q_{6}^{*}}\right\},
\end{multline*}
where $Q_{6}^{*}:= B_{6}(0)\times (\bar{\tau}_{*}+ \frac{1}{4} k^{2-p}_{j_{*}}, \bar{\tau}_{*}+ 2 k^{2-p}_{j_{*}})$, $\bar{a}^{+}_{Q_{6}^{*}}=\max\limits_{Q_{6}^{*}}\bar{a}(y,\tau).$ To estimate the last term on the right-hand side
of this inequality, we  note that by  condition ($i$)
$$\bigg(\frac{k_{j} k}{e^{\bar{\tau}_{*}}r}\bigg)^{q-p}\,\bar{a}^{+}_{Q_{6}^{*}}\leqslant \bigg(\frac{\bar{\varepsilon} k}{r}\bigg)^{q-p}
a^{+}_{Q_{6r, (6r)^{2}}(\bar{x},\bar{t})}\leqslant 1, $$
provided that $\frac{\delta}{2}r^{p}\,k^{2-p}\,e^{(\bar{\tau}_{*}+2(\frac{2^{j_{*}}}{\varepsilon})^{p-2})^{p-2}}\leqslant r^{2}$. By the first
inequality in \eqref{eq3.1} this inequality holds if $C_{*}\geqslant e^{\bar{\tau}_{*}+2(\frac{2^{j_{*}}}{\varepsilon})^{p-2}}$.

Hence
\begin{equation}\label{eq3.14}
\sup\limits_{\bar{\tau}_{*}+\frac{1}{2}k^{2-p}_{j_{*}}<\tau <\bar{\tau}_{*}+  k^{2-p}_{j_{*}}}\int\limits_{B_{4}(0)}(w-k_{j-1})_{-}^{2}\,dx+
\iint\limits_{A_{j-1}}| \nabla w|^{p} dy\,d\tau \leqslant \gamma  k^{p}_{j} |\bar{Q}^{*}\cap\{w< k_{j-1}\}|.
\end{equation}
Combining estimates \eqref{eq3.13} and \eqref{eq3.14} we obtain
\begin{equation*}
|A_{j}|^{\frac{p}{p-1}} \leqslant \gamma(\beta_{0}) |Q^{*}|^{\frac{1}{p-1}}|A_{j-1}\setminus A_{j}|.
\end{equation*}
Summing up the last inequalities over $j$, $1\leqslant j\leqslant j_{*}$, we conclude that
\begin{equation*}
j_{*}^{\frac{p-1}{p}}\,|A_{j_{*}}|\leqslant \gamma(\beta_{0}) |Q^{*}|.
\end{equation*}
Choosing $j_{*}$ by the condition
\begin{equation*}
j_{*}^{-\frac{p-1}{p}} \gamma(\beta_{0}) \leqslant \nu,
\end{equation*}
we obtain inequality \eqref{eq3.12}, which proves Lemma \ref{lem3.2}.
\end{proof}
Use Lemma \ref{lem2.6}, by \eqref{eq3.14} inequality \eqref{eq2.6} holds  for $u$ replaced by $w$, $m=p$, $k=k_{j_{*}}$ and $\eta=\gamma k^{2-p}_{j_{*}}$, therefore, we obtain
\begin{equation*}
w(y,\tau) \geqslant k_{j_{*}+1},\quad y \in B_{2}(0),
\end{equation*}
for all $\bar{\tau}_{*} + \frac{5}{8}k^{2-p}_{j_{*}} \leqslant \tau \leqslant \bar{\tau}_{*} + \frac{3}{4}k^{2-p}_{j_{*}}$.

This inequality for $u$ translates into
\begin{equation*}
u(x,t) \geqslant \frac{\varepsilon k e^{-\bar{\tau}_{*}-\frac{3}{4}(\frac{2^{j_{*}}}{\varepsilon})^{p-2}}}{2^{j_{*}+1}}=
\frac{\varepsilon^{2} k e^{- \tau_{*}-\frac{3}{4}(\frac{2^{j_{*}}}{\varepsilon})^{p-2}}}{2^{2j_{*}+1}},\quad x \in B_{2r}(\bar{x}),
\end{equation*}
for all $\bar{t}+\frac{\delta}{2}r^{p} \bigg(\dfrac{2^{j_{*}}e^{\tau_{*}+\frac{5}{8}(\frac{2^{j_{*}}}{\varepsilon})^{p-2}}}{ \varepsilon k}\bigg)^{p-2}\leqslant t \leqslant \bar{t}+\frac{\delta}{2}r^{p} \bigg(\dfrac{2^{j_{*}}e^{\tau_{*}+\frac{3}{4}(\frac{2^{j_{*}}}{\varepsilon})^{p-2}}}{ \varepsilon k}\bigg)^{p-2}$.

Choose $\sigma_{0}=\dfrac{\varepsilon^{2} e^{-\tau_{*}-\frac{3}{4}(\frac{2^{j_{*}}}{\varepsilon})^{p-2}}}{2^{2j_{*}+1}}$, by condition ($i$)
\begin{equation*}
\Big(\frac{\bar{\varepsilon} k}{r}\Big)^{p}\leqslant \varphi^{+}_{Q_{6r, (6r)^{2}}(\bar{x},\bar{t})}\bigg(\frac{\bar{\varepsilon} k}{r}\bigg)\leqslant 2 \Big(\frac{\bar{\varepsilon} k}{r}\Big)^{p},
\end{equation*}
so, from the previous we obtain
\begin{equation*}
u(x, t)\geqslant \sigma_{0} k, \quad x \in B_{2r}(\bar{x}),
\end{equation*}
for all $\bar{t}+\gamma_{1}(\tau_{*}, j_{*}, \varepsilon, \bar{\varepsilon})\,\dfrac{(\sigma_{0}k)^{2}}{\varphi^{+}_{Q_{6r, (6r)^{2}}(\bar{x},\bar{t})}(\frac{\sigma_{0}k}{r})}\leqslant t\leqslant \bar{t}+\gamma_{2}(\tau_{*}, j_{*}, \varepsilon, \bar{\varepsilon})\,\dfrac{(\sigma_{0}k)^{2}}{\varphi^{+}_{Q_{6r, (6r)^{2}}(\bar{x},\bar{t})}(\frac{\sigma_{0}k}{r})}$, \\$\gamma_{2}(\tau_{*}, j_{*}, \varepsilon, \bar{\varepsilon})=e^{\gamma(\frac{2^{j_{*}}}{\varepsilon})^{p-2}}\gamma_{1}(\tau_{*}, j_{*}, \varepsilon, \bar{\varepsilon})$, which proves Proposition~\ref{pr3.1} in the case  ($i$) with \\$b_{1}= \gamma_{1}(\tau_{*}, j_{*}, \varepsilon, \bar{\varepsilon})$ and $b_{2}= \gamma_{2}(\tau_{*}, j_{*}, \varepsilon, \bar{\varepsilon})$, by our choices numbers $\sigma_{0}$, $\tau_{*}$, $j_{*}$, $\varepsilon$, $\bar{\varepsilon}$ depend only on the data and $\beta_{0}$.

\subsection{Proof of Proposition \ref{pr3.1} in the case ($ii$)}\label{subsec.3.2}
First we prove the inequality
\begin{equation}\label{eq3.15}
a^{+}_{Q_{6r,(6r)^{2}}(\bar{x},\bar{t})} \leqslant 2 a^{-}_{Q_{6r,(6r)^{2}}(\bar{x},\bar{t})},
\end{equation}
provided that $\varepsilon_{0}$ is small enough.

Set $l:=\dfrac{s-p+2}{s-q+2} > 1$, by  condition ($A$) and the Young inequality,  using the fact that $(\alpha+p-q)\frac{l}{l-1}-p-n =(\alpha+p-q)\frac{s-p+2}{q-p} -p-n \geqslant 0$ we obtain
\begin{multline*}
\Big(a^{+}_{Q_{6r,(6r)^{2}}(\bar{x},\bar{t})}- a^{-}_{Q_{6r,(6r)^{2}}(\bar{x},\bar{t})}\Big)\frac{(\bar{\varepsilon}k)^{q-2}}{r^{q}} \leqslant
6^{\alpha} A r^{\alpha-q} (\bar{\varepsilon}k)^{q-2}= 6^{\alpha} A r^{\alpha-q} (\bar{\varepsilon}k)^{\frac{p-2}{l}+\frac{(p-2)(l-1)}{l}+q-p}\leqslant\\ \leqslant
\frac{(\bar{\varepsilon}k)^{p-2}}{4r^{p}}+\gamma r^{(\alpha+p-q)\frac{s-p+2}{q-p}-p} (\bar{\varepsilon}k)^{s}
\leqslant \frac{(\bar{\varepsilon}k)^{p-2}}{4r^{p}}+\gamma r^{n} (\bar{\varepsilon}k)^{s}.
\end{multline*}
We estimate the second term on the right-hand side of the previous inequality by using \eqref{eq3.1} and the second inequality in condition ($ii$)
\begin{multline*}
\gamma r^{n} (\bar{\varepsilon}k)^{s}\leqslant \bar{\varepsilon}^{s-q+2}\,\gamma_{0}\,\gamma\bigg(\frac{r}{\rho}\bigg)^{n}
\bigg\{\frac{(\bar{\varepsilon}k)^{p-2}}{\rho^{p}}+\,a^{+}_{Q_{6\rho,(6\rho)^{2}}(x_{0},t_{0})}\frac{(\bar{\varepsilon}k)^{q-2}}{\rho^{q}}\bigg\}\leqslant \\ \leqslant 2\bar{\varepsilon}^{s-q+2}\,\gamma_{0}\,\gamma\frac{(\bar{\varepsilon}k)^{p-2}}{\rho^{p}}\leqslant 2\bar{\varepsilon}^{s-q+2}\,\gamma_{0}\,\gamma\frac{(\bar{\varepsilon}k)^{p-2}}{r^{p}}.
\end{multline*}
Now, from the previous, using the first inequality in ($ii$)
\begin{multline*}
 \Big(a^{+}_{Q_{6r,(6r)^{2}}(\bar{x},\bar{t})}- a^{-}_{Q_{6r,(6r)^{2}}(\bar{x},\bar{t})}\Big)\frac{(\bar{\varepsilon}k)^{q-2}}{r^{q}} \leqslant
 \Big\{\frac{1}{4}+2\bar{\varepsilon}^{s-q+2}\,\gamma_{0}\,\gamma\Big\}\frac{(\bar{\varepsilon}k)^{p-2}}{r^{p}}\leqslant \\ \leqslant\Big\{\frac{1}{4}+2\bar{\varepsilon}^{s-q+2}\,\gamma_{0}\,\gamma\Big\}a^{+}_{Q_{6r,(6r)^{2}}(\bar{x},\bar{t})}\frac{(\bar{\varepsilon}k)^{q-2}}{r^{q}},
\end{multline*}
from which the required \eqref{eq3.15} follows, provided that $j_{*}$ is chosen to satisfy
\begin{equation}\label{eq3.16}
2 \bar{\varepsilon}^{s-q+2} \gamma_{0}\,\gamma\leqslant \frac{2\gamma_{0}\,\gamma}{2^{j_{*}(s-q+2)}}\leqslant\dfrac{1}{4}.
\end{equation}
 In addition, by the first inequality in condition ($ii$)
\begin{equation}\label{eq3.17}
a^{+}_{Q_{6r,(6r)^{2}}(\bar{x},\bar{t})}\bigg(\frac{ k}{r e^{\tau}}\bigg)^{q-p} \geqslant 1 \quad \text{for} \quad 0< \tau \leqslant \bar{\tau}_{*}.
\end{equation}
So,   inequality \eqref{eq3.5}  for all $ 0< \tau \leqslant \bar{\tau}_{*}\leqslant \log C_{*}$ yields
\begin{equation*}
\left|\left\{B_{r}(\bar{x}) : u\left(\cdot, \bar{t}+\frac{\delta}{2a^{+}_{Q_{6r,(6r)^{2}}(\bar{x},\bar{t})}}r^{q}\bigg(\frac{e^{\tau}}{k}\bigg)^{q-2}\right) \geqslant \varepsilon e^{-\tau} k\right\}\right| \geqslant \frac{\beta_{0}}{4} |B_{r}(\bar{x})|.
\end{equation*}
Consider the function
$$w(y,\tau):=\frac{e^{\tau}}{k} u\left(\bar{x} +r y, \bar{t}+ \frac{\delta}{2a^{+}_{Q_{6r,(6r)^{2}}(\bar{x},\bar{t})}}r^{q}\bigg(\frac{e^{\tau}}{k}\bigg)^{q-2}\right),\quad 0< \tau \leqslant \bar{\tau}_{*}.$$
The previous inequality  translates into $w$ as $\big|\big\{B_{1}(0) : w(\cdot, \tau) \geqslant \varepsilon \big\}\big| \geqslant \dfrac{\beta_{0}}{4} |B_{1}(0)|$, which implies
\begin{equation}\label{eq3.18}
\big|\big\{B_{4}(0) : w(\cdot, \tau) \geqslant \varepsilon \big\}\big| \geqslant \dfrac{\beta_{0}}{ 4^{n+1}} |B_{4}(0)|, \quad \text{for all}\quad 0< \tau \leqslant \bar{\tau}_{*}.
\end{equation}
By differentiation
\begin{equation}\label{eq3.19}
w_{\tau}= w + \frac{\delta(q-2)}{2a^{+}_{Q_{6r,(6r)^{2}}(\bar{x},\bar{t})}} r^{q}\bigg(\frac{e^{\tau}}{k}\bigg)^{q-1} u_{t} \geqslant div\, \bar{\mathbb{A}}(y,\tau,\nabla w),
\end{equation}
where $\bar{\mathbb{A}}$ satisfies the inequalities
\begin{equation}\label{eq3.20}
\begin{aligned}
& \bar{\mathbb{A}}(y,\tau,\nabla w) \nabla w \geqslant K_{1} \frac{\delta}{2} (q-2)\left\{\frac{1}{a^{+}_{Q_{6r,(6r)^{2}}(\bar{x},\bar{t})}}\bigg(\frac{e^{\tau}r}{k}\bigg)^{q-p}|\nabla w|^{p} +
\frac{\bar{a}(y,\tau)}{a^{+}_{Q_{6r,(6r)^{2}}(\bar{x},\bar{t})}}|\nabla w|^{q}\right\},
\\
&|\bar{\mathbb{A}}(y,\tau,\nabla w)| \leqslant K_{2}\frac{\delta}{2} (q-2)\left\{\frac{1}{a^{+}_{Q_{6r,(6r)^{2}}(\bar{x},\bar{t})}}\bigg(\frac{e^{\tau}r}{k}\bigg)^{q-p}|\nabla w|^{p-1} +\negthickspace \frac{\bar{a}(y,\tau)}{a^{+}_{Q_{6r,(6r)^{2}}(\bar{x},\bar{t})}}|\nabla w|^{q-1}\right\},
\end{aligned}
\end{equation}
and $\bar{a}(y,\tau):= a\bigg(\bar{x} +r y, \bar{t}+ \frac{\delta}{2a^{+}_{Q_{6r,(6r)^{2}}(\bar{x},\bar{t})}}r^{q}\bigg(\dfrac{e^{\tau}}{k}\bigg)^{q-2}\bigg).$

\begin{lemma}\label{lem3.3}
For any $\nu\in (0,1)$ there exists $j_{*}$, depending only on the data, $\beta_{0}$ and $\nu$ such that
\begin{equation}\label{eq3.21}
\left|\left\{Q^{*}: w \leqslant \frac{\varepsilon}{2^{j_{*}}} \right\}\right| \leqslant \nu |Q^{*}|,
\end{equation}
$Q^{*}= B_{4}(0)\times ( \frac{1}{2}\big(\frac{2^{j_{*}}}{\varepsilon}\big)^{q-2},  \frac{3}{4} \big(\frac{2^{j_{*}}}{\varepsilon}\big)^{q-2})$.
\end{lemma}
\begin{proof}
Using Lemma \ref{lem2.2} with $k=k_{j}:=\frac{\varepsilon}{2^{j}}$ and $l= k_{j-1}$, $1\leqslant j \leqslant j_{*}$, due to \eqref{eq3.18} we obtain
\begin{equation*}
k_{j} |A_{k_{j},4}(\tau)| \leqslant \gamma(\beta_{0})  \int\limits_{A_{k_{j-1},4}(\tau)\setminus A_{k_{j},4}(\tau)}|\nabla w| \,\,dx,\quad 0 <\tau \leqslant \bar{\tau}_{*},
\end{equation*}
where $A_{k_{j},4}(\tau):= B_{4}(0)\cap \{u(\cdot, \tau) < k_{j}\}$.

Integrating this inequality with respect to $\tau$,  $\tau\in( \frac{1}{2}k^{2-q}_{j_{*}},   \frac{3}{4}k^{2-q}_{j_{*}})$ and using the H\"{o}lder inequality
we have
\begin{equation}\label{eq3.22}
k^{\frac{q}{q-1}}_{j} |A_{j}|^{\frac{q}{q-1}} \leqslant \gamma(\beta_{0}) \left(\,\,\iint\limits_{A_{j-1}}| \nabla w|^{q} dy\,d\tau\right)^{\frac{1}{q-1}}|A_{j-1}\setminus A_{j}|,
\end{equation}
where $A_{j}:= \int\limits_{\frac{1}{2}k^{2-q}_{j_{*}}}^{\frac{3}{4}k^{2-q}_{j_{*}}}A_{k_{j},4}(\tau)\,d\tau$.

Similarly to Lemma \ref{lem3.2} 
with $|\frac{d}{d\tau}\zeta_{2}|\leqslant \gamma k^{q-2}_{j_{*}}$, by structure conditions \eqref{eq3.20} and \eqref{eq3.15} we estimate the first factor on the right-hand side of \eqref{eq3.22} as follows
\begin{multline}\label{eq3.23}
\sup\limits_{\frac{1}{2}k^{2-q}_{j_{*}}<\tau < \frac{3}{4}k^{2-q}_{j_{*}}}\int\limits_{B_{4}(0)}(w-k_{j-1})_{-}^{2}\,dx+ \frac{1}{2}\iint\limits_{A_{j-1}}| \nabla w|^{q} dy\,d\tau \leqslant\\
\leqslant \sup\limits_{\frac{1}{2}k^{2-q}_{j_{*}}<\tau <  \frac{3}{4}k^{2-q}_{j_{*}}}\int\limits_{B_{4}(0)}(w-k_{j-1})_{-}^{2}\,dx+ \iint\limits_{A_{j-1}}\frac{\bar{a}(y,\tau)}{a^{+}_{Q_{6r,(6r)^{2}}(\bar{x},\bar{t})}}| \nabla w|^{q} dy\,d\tau \leqslant \\
\leqslant\gamma k^{2}_{j-1}k^{q-2}_{j_{*}}|Q_{1}^{*}\cap\{w<k_{j-1}\}|+\gamma \frac{k^{p}_{j}}{a^{+}_{Q_{6r,(6r)^{2}}(\bar{x},\bar{t})}}\iint\limits_{Q^{*}_{6}\cap\{w<k_{j-1}\}}\bigg(\frac{e^{\tau}r}{k}\bigg)^{q-p}dy d\tau+ \\
+\gamma k^{q}_{j}\iint\limits_{Q^{*}_{6}\cap\{w<k_{j-1}\}}\frac{\bar{a}(y,\tau)}{a^{+}_{Q_{6r,(6r)^{2}}(\bar{x},\bar{t})}}dy d\tau \leqslant \gamma k^{q}_{j}|Q_{6}^{*}\cap\{w<k_{j-1}\}|+\\
+\gamma \frac{k^{p}_{j}}{a^{+}_{Q_{6r,(6r)^{2}}(\bar{x},\bar{t})}}\bigg(\frac{e^{\frac{7}{8}k^{2-q}_{j_{*}}}r}{k}\bigg)^{q-p} |Q^{*}_{6}\cap\{w<k_{j-1}\}|,
\end{multline}
where $Q^{*}_{6}:=B_{6}(0)\times (\frac{1}{4}k^{2-q}_{j_{*}},  \frac{7}{8}k^{2-q}_{j_{*}})$.
By our choices, \eqref{eq3.17} and the first inequality in ($ii$)
\begin{equation*}
\frac{k^{p}_{j}}{a^{+}_{Q_{6r,(6r)^{2}}(\bar{x},\bar{t})}}\bigg(\frac{e^{\frac{7}{8}k^{2-q}_{j_{*}}}r}{k}\bigg)^{q-p} \leqslant
\frac{k^{q}_{j}}{a^{+}_{Q_{6r,(6r)^{2}}(\bar{x},\bar{t})}}\bigg(\frac{r}{\bar{\varepsilon} k}\bigg)^{q-p}\leqslant k^{q}_{j}.
\end{equation*}
 Therefore, inequality \eqref{eq3.23} yields
\begin{equation}\label{eq3.24}
\sup\limits_{\frac{1}{2}k^{2-q}_{j_{*}}<\tau < \frac{3}{4}k^{2-q}_{j_{*}}}\int\limits_{B_{4}(0)}(w-k_{j-1})_{-}^{2}\,dx+ \iint\limits_{A_{j-1}}| \nabla w|^{q} dy\,d\tau \leqslant \gamma k^{q}_{j} |Q^{*}\cap\{w<k_{j-1}\}|.
\end{equation}
Combining \eqref{eq3.22} and \eqref{eq3.24} we arrive at
\begin{equation*}
|A_{j}|^{\frac{q}{q-1}}\leqslant \gamma(\beta_{0}) |Q^{*}|^{\frac{1}{q-1}}|A_{j-1}\setminus A_{j}|.
\end{equation*}
Summing up this inequalities in $j$, $1\leqslant j \leqslant j_{*}$ and choosing $j_{*}$ by the condition $j_{*}^{-\frac{q-1}{q}}\gamma(\beta_{0})\leqslant \nu$, we arrive at the required \eqref{eq3.21}, which proves the lemma.
\end{proof}

Use Lemma \ref{lem2.6}, similarly to that of \eqref{eq3.24} inequality \eqref{eq2.6} holds  for $u$ replaced by $w$, $m=q$, $k=k_{j_{*}}$ and $\eta=\gamma k^{2-q}_{j_{*}}$  we obtain
\begin{equation*}
w(y,\tau) \geqslant k_{j_{*}+1},\quad y \in B_{2}(0),
\end{equation*}
for all $ \frac{9}{16}k^{2-q}_{j_{*}} \leqslant \tau \leqslant  \frac{5}{8}k^{2-q}_{j_{*}}$.

This inequality for $u$ translates into
\begin{equation*}
u(x,t) \geqslant \frac{\varepsilon ke^{-\frac{5}{8}(\frac{2^{j_{*}}}{\varepsilon})^{q-2}}}{2^{j_{*}+1}},\quad x \in B_{2r}(\bar{x}),
\end{equation*}
for all \\$\bar{t}+ \dfrac{\delta}{2a^{+}_{Q_{6r,(6r)^{2}}(\bar{x},\bar{t})}}r^{q}k^{2-q} e^{\frac{9}{16}(q-2)(\frac{2^{j_{*}}}{\varepsilon})^{q-2}}
\leqslant t \leqslant \bar{t}+ \dfrac{\delta}{4a^{+}_{Q_{6r,(6r)^{2}}(\bar{x},\bar{t})}}r^{q}k^{2-q} e^{\frac{5}{8}(q-2)(\frac{2^{j_{*}}}{\varepsilon})^{q-2}}.$

Choose $\sigma_{0}=\dfrac{\varepsilon k e^{-\frac{5}{8}(\frac{2^{j_{*}}}{\varepsilon})^{q-2}}}{2^{j_{*}+1}}$,   by the first inequality in condition ($ii$)
\begin{equation*}
a^{+}_{Q_{6r, (6r)^{2}}(\bar{x},\bar{t})}\Big(\frac{\bar{\varepsilon} k}{r}\Big)^{q}\leqslant \varphi^{+}_{Q_{6r, (6r)^{2}}(\bar{x},\bar{t})}\bigg(\frac{\bar{\varepsilon} k}{r}\bigg)\leqslant 2 a^{+}_{Q_{6r, (6r)^{2}}(\bar{x},\bar{t})}\Big(\frac{\bar{\varepsilon} k}{r}\Big)^{q},
\end{equation*}
so, from the previous we obtain
\begin{equation*}
u(x, t)\geqslant \sigma_{0} k, \quad x \in B_{2r}(\bar{x}),
\end{equation*}
for all $\bar{t}+\gamma_{3}(\tau_{*}, j_{*}, \varepsilon, \bar{\varepsilon})\,\dfrac{(\sigma_{0}k)^{2}}{\varphi^{+}_{Q_{6r, (6r)^{2}}(\bar{x},\bar{t})}(\frac{\sigma_{0}k}{r})}\leqslant t\leqslant \bar{t}+\gamma_{4}(\tau_{*}, j_{*}, \varepsilon, \bar{\varepsilon})\,\dfrac{(\sigma_{0}k)^{2}}{\varphi^{+}_{Q_{6r, (6r)^{2}}(\bar{x},\bar{t})}(\frac{\sigma_{0}k}{r})}$,\\ $\gamma_{4}(\tau_{*}, j_{*}, \varepsilon, \bar{\varepsilon})=e^{\gamma(\frac{2^{j_{*}}}{\varepsilon})^{q-2}}\gamma_{3}(\tau_{*}, j_{*}, \varepsilon, \bar{\varepsilon})$, which proves Proposition~\ref{pr3.1} in the case  ($ii$) with \\$b_{1}= \gamma_{3}(\tau_{*}, j_{*}, \varepsilon, \bar{\varepsilon})$ and $b_{2}= \gamma_{4}(\tau_{*}, j_{*}, \varepsilon, \bar{\varepsilon})$, by our choices numbers $\sigma_{0}$, $\tau_{*}$, $j_{*}$, $\varepsilon$, $\bar{\varepsilon}$ depend only on the data and $\beta_{0}$.

\subsection{Proof of Proposition \ref{pr3.1} under condition \eqref{eq3.8}}\label{subsec.3.3}
To complete the proof of Proposition \ref{pr3.1}, we note that in the case \eqref{eq3.8} similarly to \eqref{eq3.15}
\begin{multline*}
a^{+}_{Q_{6\rho,(6\rho)^{2}}(x_{0},t_{0})}\frac{(\bar{\varepsilon}k)^{q-2}}{\rho^{q}}- a^{-}_{Q_{6\rho,(6\rho)^{2}}(x_{0}, t_{0})}\frac{(\bar{\varepsilon}k)^{q-2}}{\rho^{q}} \leqslant  6^{\alpha} A \rho^{\alpha-q} (\bar{\varepsilon} k)^{q-2}=\\= 6^{\alpha} A \rho^{\alpha-q} (\bar{\varepsilon} k)^{\frac{p-2}{l}+\frac{(p-2)(l-1)}{l}+q-p} 
\leqslant \frac{(\bar{\varepsilon} k)^{p-2}}{4\rho^{p}}+\gamma \rho^{(\alpha+p-q)\frac{s-p+2}{q-p}-p} (\bar{\varepsilon}k)^{s} \leqslant \\ \leqslant a^{+}_{Q_{6\rho,(6\rho)^{2}}(x_{0},t_{0})}\frac{(\bar{\varepsilon} k)^{q-2}}{4\rho^{q}}+ \bar{\varepsilon}^{s-q+2} \gamma_{0}\,\gamma \rho^{(\alpha+p-q)\frac{s-p+2}{q-p}-p-n}(\bar{\varepsilon} k)^{-2}\varphi^{+}_{Q_{6\rho,(6\rho)^{2}}(x_{0},t_{0})}\left(\frac{\bar{\varepsilon}k}{\rho}\right)  \leqslant \\ \leqslant\left(\frac{1}{4}+2\bar{\varepsilon}^{s-q+2}\gamma_{0}\,\gamma \right)a^{+}_{Q_{6\rho,(6\rho)^{2}}(x_{0},t_{0})}\frac{(\bar{\varepsilon} k)^{q-2}}{\rho^{q}},
\end{multline*}
choose $j_{*}$ from the condition \eqref{eq3.16}, from this we obtain
\begin{equation*}
a^{+}_{Q_{6r,(6r)^{2}}(\bar{x}, \bar{t})}\leqslant a^{+}_{Q_{6\rho,(6\rho)^{2}}(x_{0},t_{0})}\leqslant 2 a^{-}_{Q_{6\rho,(6\rho)^{2}}(x_{0},t_{0})}\leqslant 4 a^{-}_{Q_{6r,(6r)^{2}}(\bar{x}, \bar{t})}.
\end{equation*}
Introduce the change of variables and the new unknown function
$$w(y,\tau):=\frac{e^{\tau}}{k} u(\bar{x}+ r y, \bar{t}+\frac{\delta}{4 a^{+}_{Q_{6r,(6r)^{2}}(\bar{x},\bar{t})}}r^{q}\big(\frac{e^{\tau}}{k}\big)^{q-2}),$$
which satisfies \eqref{eq3.19}, \eqref{eq3.20}. Since condition \eqref{eq3.8} yields
\begin{equation*}
2 a^{+}_{Q_{6r,(6r)^{2}}(\bar{x},\bar{t})}\,\bigg(\frac{\bar{\varepsilon} k}{r}\bigg)^{q-p} \geqslant a^{+}_{Q_{6\rho,(6\rho)^{2}}(x_{0},t_{0})}\bigg(\frac{\bar{\varepsilon} k}{\rho}\bigg)^{q-p} \geqslant 1,
\end{equation*}
the proof of  Proposition \ref{pr3.1} under condition \eqref{eq3.8} is completely similar to that under condition ($ii$). Constant $C_{*}$ in \eqref{eq3.1} we choose by the condition
\begin{equation*}
C_{*}=e^{ \bar{\tau}_{*}+2\big(\frac{2^{j_{*}}}{\varepsilon}\big)^{p-2}}.
\end{equation*}
This completes the proof of Proposition \ref{pr3.1}.

\subsection{Expansion of positivity}\label{subsec.3.4}
Our main result of this Section reads as follows
\begin{theorem}\label{th3.1}
Let $u$ be a weak non-negative super-solution to equation \eqref{eq1.1}, let $k$ satisfies
\begin{equation}\label{eq3.25}
0 < k^{s} \leqslant  \gamma_{0} \bigg(\frac{k^{p-2}}{\rho^{n+p}}+a^{+}_{Q_{6\rho,(6\rho)^{2}}(x_{0}, t_{0})}\frac{k^{q-2}}{\rho^{n+q}}\bigg),
\end{equation}
with some $\gamma_{0}>1$, depending only on the data and let also
\begin{equation}\label{eq3.26}
\big|\big\{B_{\rho}(x_{0}): u(\cdot, \tau) > k \big\}\big|\geqslant \beta |B_{\rho}(x_{0})|,
\end{equation}
with some $\beta \in (0,1)$ and $t_{0}<\tau< t_{0}+\dfrac{k^{2}}{\varphi^{+}_{Q_{\rho,\rho^{2}}(x_{0},t_{0})}(\frac{k}{\rho})}$. Then there exist numbers $C$, $B$, $0< B_{1}\leqslant \dfrac{B_{2}}{2}$ and $\sigma_{1}\in (0,1)$ depending only on the data  such that  either
\begin{equation}\label{eq3.27}
\beta^{B} k\leqslant C\rho,
\end{equation}
or
\begin{equation}\label{3.28}
u(x,t) \geqslant \sigma_{1}\beta^{B} k, \quad x\in B_{2\rho}(x_{0}),
\end{equation}
and for all
\begin{equation}\label{eq3.29}
t_{0}+B_{1}\frac{(\sigma_{1}\beta^{B} k)^{2}}{\varphi^{+}_{Q_{14\rho,(14\rho)^{2}}(x_{0},t_{0})}(\frac{\sigma_{1}\beta^{B}k}{\rho})}\leqslant t \leqslant t_{0}+
B_{2}\frac{(\sigma_{1}\beta^{B}k)^{2}}{\varphi^{+}_{Q_{14\rho,(14\rho)^{2}}(x_{0},t_{0})}(\frac{\sigma_{1}\beta^{B}k}{\rho})},
\end{equation}
provided that $ Q_{16\rho, (16\rho)^{2}}(x_{0},t_{0})\subset \Omega_{T}$.
\end{theorem}
\begin{proof}
In what follows, we assume that
\begin{equation}\label{eq3.30}
\beta^{B} k \geqslant C\rho,
\end{equation}
$C\geqslant C_{*},$ where $C_{*}$ is the constant from the first inequality in \eqref{eq3.1}.
Condition \eqref{eq3.25} and Lemma \ref{lem3.1} (see \eqref{eq3.6}) yield
\begin{equation}\label{eq3.31}
\big|\big\{B_{\rho}(x_{0}): u(\cdot, t) > \frac{\beta}{8} k \big\}\big|\geqslant \frac{\beta}{4} |B_{\rho}(x_{0})|,
\end{equation}
for all $\tau < t \leqslant \tau + \dfrac{\delta k^{2}}{\varphi^{+}_{Q_{6\rho,(6\rho)^{2}}(x_{0}, \tau)}(\frac{k}{\rho})}$,\quad $\delta=\dfrac{\beta^{q+1}}{\gamma}$.
\end{proof}
Write down the energy estimates \eqref{eq2.3} with $k$ replaced by $\frac{\beta}{8} k$, for the pair of cylinders
$Q:= B_{\rho}(x_{0})\times (\tau + \dfrac{\eta}{2}, \tau + \eta)$, $Q_{1}:= B_{2\rho}(x_{0})\times (\tau , \tau + \eta)$, $\eta=\dfrac{\delta k^{2}}{\varphi^{+}_{Q_{6\rho,(6\rho)^{2}}(x_{0},\tau)}(\frac{k}{\rho})}$ and take $\left|\dfrac{d}{dt}\zeta_{2}\right| \leqslant \gamma \dfrac{\varphi^{+}_{Q_{6\rho,(6\rho)^{2}}(x_{0},\tau)}(\frac{k}{\rho})}{\delta k^{2}}$ and $|\nabla \zeta_{1}| \leqslant \dfrac{\gamma}{\rho}$.
By condition ($A$) and \eqref{eq3.25} we have
\begin{multline*}
\iint\limits_{Q}|\nabla \big(u-\frac{\beta}{8} k\big)_{-}|^{p} dx\,dt \leqslant 
\frac{\gamma}{\beta^{q+1}}  \bigg(\frac{\beta k}{\rho}\bigg)^{p}\frac{\bigg(1+a^{+}_{Q_{6\rho,(6\rho)^{2}}(x_{0},\tau)}\big(\frac{\beta k}{8\rho}\big)^{q-p}\bigg)}{\bigg(1+a^{-}_{Q_{6\rho,(6\rho)^{2}}(x_{0},\tau)}\big(\frac{\beta k}{8\rho}\big)^{q-p}\bigg)}\, |Q|\leqslant  \frac{\gamma}{\beta^{q+1}}  \bigg(\frac{\beta k}{\rho}\bigg)^{p}|Q|.
\end{multline*}
From this and \eqref{eq3.31} it follows that there exists $t_{1}\in(\tau + \dfrac{\eta}{2}, \tau+\eta )$ such that
\begin{multline}\label{eq3.32}
\int\limits_{B_{\rho}(x_{0})\times\{t_{1}\}} |\nabla\big(u-\frac{\beta}{8} k\big)_{-}| dx \leqslant \frac{\gamma}{\beta^{\frac{q+1}{p}}} \beta k \rho^{n-1}
\,\, \text{and}\,\, \big|\big\{B_{\rho}(x_{0}): u(\cdot, t_{1}) > \frac{\beta}{8} k \big\}\big|\geqslant \frac{\beta}{4} |B_{\rho}(x_{0})|.
\end{multline}
The local clustering Lemma \ref{lem2.1} with $\mathcal{K}=\frac{\gamma}{\beta^{\frac{q+1}{p}}}$,  $\alpha=\frac{\beta}{4}$, $\nu=\frac{1}{2}$, $\xi=\frac{1}{2}$ and $k$ replaced by $\frac{\beta}{8} k$ yields
\begin{equation}\label{eq3.33}
\left|\left\{B_{r}(\bar{x}): u(\cdot, t_{1}) > \frac{\beta}{16} k \right\}\right|\geqslant \frac{1}{2} |B_{r}(\bar{x})|,
\quad r=\epsilon\beta^{2+\frac{q+1}{p}} \rho
\end{equation}
with some $\bar{x} \in B_{\rho}(y)$ and some $\epsilon \in (0,1)$ depending only on the data.

Proposition \ref{pr3.1} with $\beta_{0}= \frac{1}{2}$ and $k$ replaced by $\frac{\beta}{16} k$ implies
\begin{equation*}
u(x,t) \geqslant \sigma_{0} \beta k,\quad x\in B_{2r}(\bar{x}),
\end{equation*}
for all
$$t_{2}:= t_{1}+b_{1}\frac{(\sigma_{0} \beta k)^{2}}{\varphi^{+}_{Q_{6r,(6r)^{2}}(\bar{x},t_{1})}(\frac{\sigma_{0} \beta k}{r})}\leqslant t \leqslant  t_{1}+b_{2}\frac{(\sigma_{0} \beta k)^{2}}{\varphi^{+}_{Q_{6r,(6r)^{2}}(\bar{x},t_{1})}(\frac{\sigma_{0} \beta k}{r})},$$
with some $\sigma_{0} \in (0,1)$ and $b_{1}$, $b_{2}>0$ depending only on the data. From this by iteration
we obtain
\begin{equation}\label{eq3.34}
u(x,t)\geqslant \sigma^{j}_{0}\beta k,\quad x \in B_{2^{j}r}(\bar{x}),
\end{equation}
for all
\begin{equation}\label{eq3.35}
t_{j+1}:= t_{j}+b_{1}\frac{(\sigma^{j}_{0} \beta k)^{2}}{\varphi^{+}_{Q_{2^{j}6r,(2^{j}6r)^{2}}(\bar{x},t_{j})}(\frac{\sigma^{j}_{0} \beta k}{2^{j}r})}\leqslant t \leqslant
t_{j}+b_{2}\frac{(\sigma^{j}_{0} \beta k)^{2}}{\varphi^{+}_{Q_{2^{j}6r,(2^{j}6r)^{2}}(\bar{x},t_{j})}(\frac{\sigma^{j}_{0} \beta k}{2^{j}r})}.
\end{equation}
Choosing $j$ by the condition $2^{j}r=2\rho$, from \eqref{eq3.34} we obtain
\begin{equation}\label{eq3.36}
u(x,t)\geqslant  \frac{k}{\gamma(\sigma_{0})}\,\beta^{1+(2+\frac{q+1}{p}) \log\frac{1}{\sigma_{0}}}=\sigma_{1} \beta^{B} k,\quad x\in B_{2\rho}(x_{0}),
\end{equation}
for all $t$ satisfying \eqref{eq3.35}. 

By our choices we have
\begin{equation*}
\varphi^{+}_{Q_{2^{j}6r,(2^{j}6r)^{2}}(\bar{x},t_{j})}\Big(\frac{\sigma^{j}_{0} \beta k}{2^{j}r}\Big)\leqslant 
\varphi^{+}_{Q_{12\rho,(12\rho)^{2}}(\bar{x},t_{j})}\Big(\frac{\sigma_{1} \beta^{B} k}{2\rho}\Big)\leqslant
\varphi^{+}_{Q_{14\rho,(14\rho)^{2}}(x_{0}, t_{0})}\Big(\frac{\sigma_{1} \beta^{B} k}{2\rho}\Big),
\end{equation*}
and hence
\begin{equation}\label{eq3.37}
t_{j}+b_{2}\frac{(\sigma^{j}_{0} \beta k)^{2}}{\varphi^{+}_{Q_{2^{j}6r,(2^{j}6r)^{2}}(\bar{x},t_{j})}(\frac{\sigma^{j}_{0} \beta k}{2^{j}r})}\geqslant t_{0}+b_{2}\frac{(\sigma_{1}\beta^{B}  k)^{2}}{\varphi^{+}_{Q_{14 \rho,(14 \rho)^{2}}(x_{0}, t_{0})}(\frac{\sigma_{1} \beta^{B} k}{2 \rho})}.
\end{equation}
In addition, for any $i\leqslant j$
\begin{multline}\label{eq3.38}
\varphi^{+}_{Q_{2^{i}6r,(2^{i}6r)^{2}}(\bar{x},t_{i})}\Big(\frac{\sigma^{i}_{0} \beta k}{2^{i}r}\Big)\geqslant
\varphi^{-}_{Q_{2^{j}6r,(2^{j}6r)^{2}}(\bar{x},t_{j})}\Big(\frac{\sigma^{i}_{0} \beta k}{2^{i}r}\Big)\geqslant\\ \geqslant
\varphi^{-}_{Q_{14 \rho,(14 \rho)^{2}}(x_{0}, t_{0})}\Big(\frac{\sigma^{j}_{0} \beta k}{2^{j}r}\Big)\Big(\frac{\sigma_{0}}{2}\Big)^{p(i-j)}=
\varphi^{-}_{Q_{14 \rho,(14 \rho)^{2}}(x_{0}, t_{0})}\Big(\frac{\sigma_{1} \beta^{B} k}{2\rho}\Big)\Big(\frac{\sigma_{0}}{2}\Big)^{p(i-j)}.
\end{multline}
Let us estimate the term on the right-hand side of \eqref{eq3.38}. If $a^{+}_{Q_{14 \rho,(14 \rho)^{2}}(x_{0}, t_{0})}\Big(\dfrac{\sigma_{1} \beta^{B} k}{2\rho}\Big)^{q-p}\leqslant 1$, then
\begin{equation}\label{eq3.39}
\varphi^{-}_{Q_{14 \rho,(14 \rho)^{2}}(x_{0}, t_{0})}\Big(\frac{\sigma_{1} \beta^{B} k}{2\rho}\Big)\geqslant \Big(\frac{\sigma_{1} \beta^{B} k}{2\rho}\Big)^{p}\geqslant \frac{1}{2}\varphi^{+}_{Q_{14 \rho,(14 \rho)^{2}}(x_{0}, t_{0})}\Big(\frac{\sigma_{1} \beta^{B} k}{2\rho}\Big).
\end{equation}
On the other hand, if $a^{+}_{Q_{14 \rho,(14 \rho)^{2}}(x_{0}, t_{0})}\Big(\dfrac{\sigma_{1} \beta^{B} k}{2\rho}\Big)^{q-p}\geqslant 1$,
using condition ($A$) and \eqref{eq3.25}, similarly to \eqref{eq3.15} we obtain that 
\begin{equation*}
a^{+}_{Q_{14 \rho,(14 \rho)^{2}}(x_{0}, t_{0})}\leqslant 2a^{-}_{Q_{14 \rho,(14 \rho)^{2}}(x_{0}, t_{0})},
\end{equation*}
provided that $\sigma_{1}$ satisfies \eqref{eq3.16}, i.e.
\begin{equation*}
(\sigma_{1}\beta^{B})^{s-q+2}\gamma_{0}\,\gamma\leqslant \sigma_{1}^{s-q+2}\gamma_{0}\,\gamma \leqslant \frac{1}{4},
\end{equation*}
which holds by possible reducing of $\sigma_{0}$, if needed. Therefore
\begin{equation}\label{eq3.40}
\varphi^{-}_{Q_{14 \rho,(14 \rho)^{2}}(x_{0}, t_{0})}\Big(\frac{\sigma_{1} \beta^{B} k}{2\rho}\Big)\geqslant a^{-}_{Q_{14 \rho,(14 \rho)^{2}}(x_{0}, t_{0})}\Big(\frac{\sigma_{1} \beta^{B} k}{2\rho}\Big)^{q} \geqslant \frac{1}{2}\varphi^{+}_{Q_{14 \rho,(14 \rho)^{2}}(x_{0}, t_{0})}\Big(\frac{\sigma_{1} \beta^{B} k}{2\rho}\Big).
\end{equation}
Collecting \eqref{eq3.38}--\eqref{eq3.40}, we obtain
\begin{equation*}
\varphi^{+}_{Q_{2^{i}6r,(2^{i}6r)^{2}}(\bar{x},t_{i})}\Big(\frac{\sigma^{i}_{0} \beta k}{2^{i}r}\Big)\geqslant
\frac{1}{2}\varphi^{+}_{Q_{14 \rho,(14 \rho)^{2}}(x_{0}, t_{0})}\Big(\frac{\sigma_{1} \beta^{B} k}{2\rho}\Big)\Big(\frac{\sigma_{0}}{2}\Big)^{p(i-j)},\quad i\leqslant j.
\end{equation*}
Similarly to \eqref{eq3.39}, \eqref{eq3.40} we also have
\begin{equation*}
k^{-2}\,\varphi^{+}_{Q_{\rho,\rho^{2}}(x_{0}, t_{0})}\Big(\frac{ k}{\rho}\Big)\geqslant 
(\sigma_{1}\beta^{B}k)^{-2}\,\varphi^{+}_{Q_{\rho,\rho^{2}}(x_{0}, t_{0})}\Big(\frac{\sigma_{1}\beta^{B} k}{\rho}\Big)\geqslant
\frac{1}{2} (\sigma_{1}\beta^{B}k)^{-2}\varphi^{+}_{Q_{14\rho,(14\rho)^{2}}(x_{0}, t_{0})}\Big(\frac{\sigma_{1}\beta^{B} k}{\rho}\Big).
\end{equation*}
 So,
\begin{multline}\label{eq3.41}
t_{j+1}\leqslant t_{0}+\frac{k^{2}}{\varphi^{+}_{Q_{\rho,\rho^{2}}(x_{0}, t_{0})}\big(\frac{ k}{\rho}\big)}+b_{1}\sum\limits_{i=0}^{j}\frac{(\sigma_{0}^{i}\beta k)^{2}}{\varphi^{+}_{Q_{2^{i}6r,(2^{i}6r)^{2}}(\bar{x}, t_{i})}\big(\frac{\sigma_{0}^{i}\beta k}{2^{i}r}\big)}\leqslant \\ \leqslant t_{0}+ 2\frac{(\sigma_{1}\beta^{B}k)^{2}}{\varphi^{+}_{Q_{14\rho,(14\rho)^{2}}(x_{0}, t_{0})}\big(\frac{\sigma_{1}\beta^{B} k}{\rho}\big)}\bigg\{1+b_{1}\sum\limits_{i=0}^{j}\bigg(\frac{2^{p}}{\sigma_{0}^{p-2}}\bigg)^{i-j}\bigg\}\leqslant\\ \leqslant 
t_{0}+2(1+b_{1}\,\gamma)\frac{(\sigma_{1}\beta^{B}k)^{2}}{\varphi^{+}_{Q_{14\rho,(14\rho)^{2}}(x_{0}, t_{0})}\big(\frac{\sigma_{1}\beta^{B} k}{\rho}\big)}.
\end{multline}
Inequalities \eqref{eq3.35}--\eqref{eq3.37} and \eqref{eq3.41} yield
\begin{equation*}
u(x,t)\geqslant \sigma_{1} \beta^{B} k,\quad x\in B_{2\rho}(x_{0}),
\end{equation*}
for all $t_{0}+2(1+b_{1}\,\gamma)\dfrac{(\sigma_{1}\beta^{B}k)^{2}}{\varphi^{+}_{Q_{14\rho,(14\rho)^{2}}(x_{0}, t_{0})}\big(\frac{\sigma_{1}\beta^{B} k}{\rho}\big)}\leqslant t\leqslant t_{0}+b_{2}\dfrac{(\sigma_{1}\beta^{B}  k)^{2}}{\varphi^{+}_{Q_{14 \rho,(14 \rho)^{2}}(x_{0}, t_{0})}(\frac{\sigma_{1} \beta^{B} k}{\rho})}.$
This completes the proof of Theorem \ref{th3.1} with $B_{1}=2(1+b_{1}\,\gamma)$ and $B_{2}=b_{2}$.

\section{Weak Harnack inequality, proof of Theorem \ref{th1.1}}\label{Sec4}

Fix $\xi_{0} \in (0,1)$ depending only on the data to be chosen later. Following Kuusi \cite{Kuu},  we will distinguish two alternative cases:

''Hot'' alternative case: either there exist a time level $\bar{t} \in \left(t_{0}, t_{0}+\frac{\mathcal{I}^{2}}{\varphi^{+}_{{Q_{2\rho,(2\rho)^{2}}(x_{0},t_{0})}}(\frac{\mathcal{I}}{\rho})}\right)$ and a number $\lambda_{0}>1$ such that
\begin{equation}\label{eq4.1}
\big|\big\{B_{2\rho}(x_{0}): u(\cdot, \bar{t}) \geqslant \lambda_{0} \mathcal{I} \big\}\big| \geqslant \lambda^{-\frac{\xi_{0}}{B}}_{0} |B_{2\rho}(x_{0})|,
\end{equation}
or such inequality is violated, i.e. 

''Cold'' alternative case: for all $t\in \left(t_{0}, t_{0}+\frac{\mathcal{I}^{2}}{\varphi^{+}_{{Q_{2\rho,(2\rho)^{2}}(x_{0},t_{0})}}(\frac{\mathcal{I}}{\rho})} \right)$ and for any $\lambda >1$ there holds
\begin{equation}\label{eq4.2}
\big|\big\{B_{2\rho}(x_{0}): u(\cdot, t) \geqslant \lambda \mathcal{I} \big\}\big| \leqslant \lambda^{-\frac{\xi_{0}}{B}} |B_{2\rho}(x_{0})|,
\end{equation}
here $B>1$ is the number defined in Theorem \ref{th3.1} and $\mathcal{I}=\fint\limits_{B_{\rho}(x_{0})} u(x, t_{0}) dx$.

\subsection{Proof of Theorem \ref{th1.1} under the ''hot'' alternative case, condition \eqref{eq4.1}}\label{subsect4.1}
We will use Theorem \ref{th3.1}, so we need to obtain estimate \eqref{eq3.25}, i.e. the following lemma
\begin{lemma}\label{lem4.1}
Let \eqref{eq4.1} hold then 
\begin{equation}\label{eq4.3}
\big(\lambda_{0}^{\xi_{0}} \mathcal{I}\big)^{s} \leqslant \gamma d^{s} \bigg\{\frac{ (\lambda_{0}^{\xi_{0}}\mathcal{I})^{p-2}}{\rho^{n+p}}+a^{+}_{Q_{6\rho,(6\rho)^{2}}(x_{0},t_{0})}\frac{ (\lambda_{0}^{\xi_{0}}\mathcal{I})^{q-2}}{\rho^{n+q}}\bigg\},
\end{equation}
provided that $\mathcal{I}\geqslant C_{*} \rho$. Here $C_{*}$ is the constant defined in \eqref{eq3.1}.
\end{lemma}
\begin{proof}
Inequality $\mathcal{I}\geqslant C_{*} \rho$ yields $\lambda_{0}\mathcal{I} \geqslant C_{*} \rho$, so by Lemma \ref{lem3.1} and conditions \eqref{eq3.6} 
\begin{equation*}
\big|\big\{B_{2\rho}(x_{0}): u(\cdot, t) \geqslant \frac{1}{8}\lambda_{0}^{1-\frac{\xi_{0}}{B}} \mathcal{I} \big\}\big| \geqslant \frac{1}{2}\lambda^{-\frac{\xi_{0}}{B}}_{0} |B_{2\rho}(x_{0})|,
\end{equation*}
for all $t \in( \bar{t}, \bar{t} +\eta)$, $\eta=\gamma^{-1}\dfrac{\lambda_{0}^{-\frac{(q+1)\xi_{0}}{B}}(\lambda_{0}\mathcal{I})^{2}}{\varphi^{+}_{Q_{4\rho,(4\rho)^{2}}(x_{0},\bar{t})}(\frac{\lambda_{0} \mathcal{I}}{\rho})}$.

From this
\begin{equation}\label{eq4.4}
\frac{1}{16} \lambda_{0}^{1-\frac{\xi_{0}}{B}-\frac{\xi_{0}}{Bs}}
|B_{2\rho}(x_{0})|^{\frac{1}{s}}\eta^{\frac{1}{s}} \mathcal{I} \leqslant \,\Biggr(\,\,\iint\limits_{Q_{2\rho,\eta}(x_{0},\bar{t})} u^{s} dx\,d\tau\Biggr)^{\frac{1}{s}}\leqslant d.
\end{equation}
If $a^{+}_{Q_{4\rho,(4\rho)^{2}}(x_{0},\bar{t})}\bigg(\dfrac{\lambda_{0} \mathcal{I}}{\rho}\bigg)^{q-p}\geqslant 1$, then \eqref{eq4.4}
implies
\begin{equation*}
\lambda_{0}^{1-\frac{\xi_{0}(s+q+2)}{B(s-q+2)}}\,\mathcal{I}\leqslant \gamma \big[a^{+}_{Q_{4\rho,(4\rho)^{2}}(x_{0},\bar{t})}\big]^{\frac{1}{s-q+2}}\,\rho^{-\frac{n+q}{s-q+2}} d^{\frac{s}{s-q+2}},
\end{equation*}
and if $1-\frac{\xi_{0}(s+q+2)}{B(s-q+2)} \geqslant \xi_{0} $, i.e. $\xi_{0}(1+\frac{s+q+2}{B(s-q+2)} ) \leqslant 1$, then
\begin{equation*}
\lambda_{0}^{\xi_{0}}\,\mathcal{I} \leqslant \gamma \big[a^{+}_{Q_{4\rho,(4\rho)^{2}}(x_{0},\bar{t})}\big]^{\frac{1}{s-q+2}} d^{\frac{s}{s-q+2}} \rho^{-\frac{n+q}{s-q+2}}.
\end{equation*}
By the fact that $a^{+}_{Q_{4\rho,(4\rho)^{2}}(x_{0},\bar{t})}\leqslant a^{+}_{Q_{6\rho,(6\rho)^{2}}(x_{0},t_{0})}$, this 
 proves inequality \eqref{eq4.3} in the case \\$a^{+}_{Q_{4\rho,(4\rho)^{2}}(x_{0},\bar{t})}\bigg(\dfrac{\lambda_{0} \mathcal{I}}{\rho}\bigg)^{q-p}\geqslant 1$.

And if $a^{+}_{Q_{4\rho,(4\rho)^{2}}(x_{0},\bar{t})}\bigg(\dfrac{\lambda_{0} \mathcal{I}}{\rho}\bigg)^{q-p}\leqslant 1$, then \eqref{eq4.4} yields
\begin{equation*}
\lambda_{0}^{\xi_{0}}\,\mathcal{I}\leqslant \lambda_{0}^{1-\frac{\xi_{0}(s+q+2)}{B(s-p+2)}}\,\mathcal{I}\leqslant \gamma \rho^{-\frac{n+p}{s-p+2}} d^{\frac{s}{s-p+2}},
\end{equation*}
provided that  $1-\frac{\xi_{0}(s+q+2)}{B(s-p+2)} \geqslant \xi_{0} $, i.e. $\xi_{0}(1+\frac{s+q+2}{B(s-p+2)} ) \leqslant 1$. The last inequality proves \eqref{eq4.3} in the case $a^{+}_{Q_{4\rho,(4\rho)^{2}}(x_{0},\bar{t})}\bigg(\dfrac{\lambda_{0} \mathcal{I}}{\rho}\bigg)^{q-p}\leqslant 1$. This completes the proof of the lemma.
\end{proof}

 We use  Theorem \ref{th3.1} with $k=\lambda_{0}^{\xi_{0}} \mathcal{I}$, $\beta=\frac{1}{2}\lambda_{0}^{-\frac{\xi_{0}}{B}}$ and $\tau= \bar{t}$. By Lemma \ref{lem4.1}, inequality \eqref{eq3.25} with $\gamma_{0}=\gamma d^{s}$ of Theorem \ref{th3.1} holds, so, using the fact that $\beta^{B}\,k=\dfrac{\mathcal{I}}{2^{B}}$, we obtain that either $\mathcal{I} \leqslant C\,2^{B}\rho$, or
\begin{equation}\label{eq4.5}
u(x,t)\geqslant \frac{\sigma_{1}}{2^{B}}\,\mathcal{I}=\bar{\sigma}_{1}\,\mathcal{I},\quad x\in B_{4\rho}(x_{0}),
\end{equation}
for all $t_{0}+B_{1}\dfrac{(\bar{\sigma}_{1}\mathcal{I})^{2}}{\varphi^{+}_{Q_{14\rho,(14\rho)^{2}}(x_{0},t_{0})}( \frac{\bar{\sigma}_{1}\mathcal{I}}{\rho})} \leqslant t \leqslant t_{0}+B_{2}\dfrac{(\bar{\sigma}_{1}\mathcal{I})^{2}}{\varphi^{+}_{Q_{14\rho,(14\rho)^{2}}(x_{0},t_{0})}( \frac{\bar{\sigma}_{1}\mathcal{I}}{\rho})}$, which proves Theorem \ref{th1.1} under the ''hot'' alternative case, condition \eqref{eq4.1}.

\subsection{Proof of Theorem \ref{th1.1} under the ''cold'' alternative case, condition \eqref{eq4.2}}\label{subsect4.2}
In what follows, we will assume that
\begin{equation}\label{eq4.6}
\mathcal{I} \geqslant C_{1} \biggr\{\rho + \rho\,\,\psi^{-1}_{Q_{14\rho,(14\rho)^{2}}(x_{0},t_{0})}\bigg(\frac{\rho^{2}}{T-t_{0}}\bigg)\biggr\},
\end{equation}
with some positive $C_{1} >0$ to be chosen later.

First we note that condition \eqref{eq4.2} yields
\begin{multline}\label{eq4.7}
\fint\limits_{B_{2\rho}(x_{0})}u(x,t)^{\varkappa}\,dx= \frac{\varkappa}{|B_{2\rho}(x_{0})|}\int\limits_{0}^{\infty}
|\big\{u > \lambda\big\}| \lambda^{\varkappa-1}d \lambda =\\=\frac{\varkappa \mathcal{I}^{\varkappa}}{|B_{2\rho}(x_{0})|}\int\limits_{0}^{\infty}|\big\{B_{2\rho}(x_{0}) :u > \lambda \mathcal{I}\big\}| \lambda^{\varkappa-1}d \lambda \leqslant \\ \leqslant\mathcal{I}^{\varkappa}
+ \varkappa \mathcal{I}^{\varkappa} \int\limits_{1}^{\infty} \lambda^{\varkappa -\frac{\xi_{0}}{B} -1} d\lambda \leqslant 3 \mathcal{I}^{\varkappa},\qquad \varkappa=\dfrac{\xi_{0}}{2B},
\end{multline}
for all $t \in \bigg(t_{0}, t_{0}+\dfrac{\mathcal{I}^{2}}{\varphi^{+}_{{Q_{2\rho,(2\rho)^{2}}(x_{0},t_{0})}}(\frac{\mathcal{I}}{\rho})}\bigg)$.

The following  lemma is the  uniform upper bound for the super-solutions, which further gives us the possibility to use the expansion of positivity theorem, Theorem \ref{th3.1}.

\begin{lemma}\label{lem4.2}
Fix $l$ by the condition
\begin{equation}\label{eq4.8}
1<l:= \frac{s-p+2}{s-q+2} <\frac{n+p}{n}.
\end{equation}
Then for all $m$ in the range
\begin{equation}\label{eq4.9}
q-2 < m < q-1+ \frac{p-n(l-1)}{l n}
\end{equation}
there holds
\begin{multline}\label{eq4.10}
\frac{\mathcal{I}^{p-2}}{\rho^{n+p}}\iint\limits_{Q_{\frac{7}{4}\rho, \frac{3}{4}\eta}(x_{0},t_{0})} \left(\frac{u}{\mathcal{I}}+1\right)^{m-q+p} dx\, dt +\\+a^{+}_{Q_{2\rho, (2\rho)^{2}}(x_{0},t_{0})}\frac{\mathcal{I}^{q-2}}{\rho^{n+q}}\iint\limits_{Q_{\frac{7}{4}\rho, \frac{3}{4}\eta}(x_{0},t_{0})} \left(\frac{u}{\mathcal{I}}+1\right)^{m} dx\, dt\leqslant \gamma,
\end{multline}
where $\eta=\dfrac{\mathcal{I}^{2}}{\varphi^{+}_{Q_{2\rho, (2\rho)^{2}}(x_{0},t_{0})}\big(\frac{\mathcal{I}}{\rho}\big)}$.
\end{lemma}
\begin{proof}
Fix $\sigma \in (0,1)$, let $\dfrac{15}{8}\rho<(1-\sigma)r < r < 2\rho$ and let $\zeta(x)\in C^{1}_{0}(B_{r}(x_{0}))$, $0\leqslant \zeta(x) \leqslant 1$, $\zeta(x)=1$ in $B_{(1-\sigma)r}(x_{0})$ and $| \nabla \zeta(x)| \leqslant \dfrac{1}{\sigma r}$. We use Lemma \ref{lem2.4} with
$\varepsilon=1-\frac{\varkappa}{l}$, where $\varkappa$ is the number defined in \eqref{eq4.7}. By the Sobolev embedding theorem and by \eqref{eq4.7}
we obtain
\begin{multline*}
\iint\limits_{Q_{r,\eta}(x_{0}, t_{0})} (u+\mathcal{I})^{p-2+\varkappa \frac{n+p}{ln}} \zeta^{q}(x) dx\, dt \leqslant \\ \leqslant
\gamma\, \bigg(\sup\limits_{t_{0}<t< t_{0}+\eta}\int\limits_{B_{r}(x_{0})}(u+\mathcal{I})^{\frac{\varkappa}{l}} dx\bigg)^{\frac{p}{n}}\iint\limits_{Q_{r,\eta}(x_{0}, t_{0})}(u+\mathcal{I})^{-2 +\frac{\varkappa}{l}} |\nabla (u\zeta^{\frac{q}{p}}(x)|^{p} dx\, dt
\leqslant \\ \leqslant \gamma \rho^{p} \mathcal{I}^{\frac{\varkappa p}{ln}}\iint\limits_{Q_{r,\eta}(x_{0}, t_{0})}(u+\mathcal{I})^{-2 +\frac{\varkappa}{l}} |\nabla (u\zeta^{\frac{q}{p}}(x)|^{p} dx\, dt \leqslant \gamma \sigma^{-q}\rho^{p} \mathcal{I}^{\frac{\varkappa p}{ln}} \times \\ \times\left\{\frac{1}{\rho^{p}}\iint\limits_{Q_{r,\eta}(x_{0}, t_{0})} (u+\mathcal{I})^{p-2+\frac{\varkappa}{l}}dx\,dt +\frac{a^{+}_{Q_{2\rho,(2\rho)^{2}}(x_{0}, t_{0})}}{\rho^{q}}\iint\limits_{Q_{r,\eta}(x_{0}, t_{0})} (u+\mathcal{I})^{q-2+\frac{\varkappa}{l}}dx\,dt\right\},
\end{multline*}
which  yields
\begin{multline}\label{eq4.11}
\frac{\mathcal{I}^{p-2}}{\rho^{n+p}}\iint\limits_{Q_{r,\eta}(x_{0}, t_{0})} \left(\frac{u}{\mathcal{I}}+1\right)^{p-2+\varkappa \frac{n+p}{ln}}\zeta^{q}(x) dx\, dt \leqslant \\ \leqslant
 \gamma \sigma^{-q}\left\{\frac{\mathcal{I}^{p-2}}{\rho^{n+p}}\iint\limits_{Q_{r,\eta}(x_{0}, t_{0})} \left(\frac{u}{\mathcal{I}}+1\right)^{p-2+\frac{\varkappa}{l}} dxdt + \right.\\
\left. +a^{+}_{Q_{2\rho, (2\rho)^{2}}(x_{0},t_{0})}\frac{\mathcal{I}^{q-2}}{\rho^{n+q}}\iint\limits_{Q_{r,\eta}(x_{0}, t_{0})} \left(\frac{u}{\mathcal{I}}+1\right)^{q-2+\frac{\varkappa}{l}} dxdt\right\}.
\end{multline}
By condition ($A$) we have
\begin{multline}\label{eq4.12}
a^{+}_{Q_{2\rho,(2\rho)^{2}}(x_{0}, t_{0})}\iint\limits_{Q_{r,\eta}(x_{0}, t_{0})} (u+\mathcal{I})^{q-2+\varkappa\frac{n+p}{ln}}\zeta^{q}(x)dxdt \leqslant \\ \leqslant \gamma a^{+}_{Q_{2\rho,(2\rho)^{2}}(x_{0}, t_{0})}\,\,|Q_{r,\eta}(x_{0}, t_{0})|\,\,\mathcal{I}^{q-2+\varkappa\frac{n+p}{ln}}+\\+a^{-}_{Q_{2\rho,(2\rho)^{2}}(x_{0}, t_{0})}\iint\limits_{Q_{r,\eta}(x_{0}, t_{0})} \negthickspace u^{q-2+\varkappa\frac{n+p}{ln}}\zeta^{q}(x)dxdt+\\+ \gamma \rho^{\alpha}\iint\limits_{Q_{r,\eta}(x_{0}, t_{0})} \negthickspace u^{q-2+\varkappa\frac{n+p}{ln}}\zeta^{q}(x)dxdt.
\end{multline}
Let us estimate the terms on the right-hand side of \eqref{eq4.12}. Similarly to \eqref{eq4.11}
\begin{multline*}
a^{-}_{Q_{2\rho,(2\rho)^{2}}(x_{0}, t_{0})}\iint\limits_{Q_{r,\eta}(x_{0}, t_{0})} (u+\mathcal{I})^{q-2+\varkappa\frac{n+p}{ln}}dxdt \leqslant \\ \leqslant \gamma \rho^{q-p} \bigg(\sup\limits_{t_{0}<t< t_{0}+\eta}\int\limits_{B_{r}(x_{0})}(u+\mathcal{I})^{\frac{\varkappa}{l}} dx\bigg)^{\frac{p}{n}} \times\\\times\iint\limits_{Q_{r,\eta}(x_{0}, t_{0})}a(x,t)(u+\mathcal{I})^{-2 +\frac{\varkappa}{l}} |\nabla (u\zeta(x))|^{q} dx\, dt \leqslant \\  \leqslant \gamma \rho^{q} \mathcal{I}^{\frac{\varkappa p}{ln}}\iint\limits_{Q_{r,\eta}(x_{0}, t_{0})}a(x,t) (u+\mathcal{I})^{-2 +\frac{\varkappa}{l}} |\nabla (u\zeta(x))|^{q} dx\, dt \leqslant \gamma \sigma^{-q}\rho^{q} \mathcal{I}^{\frac{\varkappa p}{ln}} \times\\\times \left\{\frac{1}{\rho^{p}}\iint\limits_{Q_{r,\eta}(x_{0}, t_{0})} (u+\mathcal{I})^{p-2+\frac{\varkappa}{l}}dxdt +\frac{a^{+}_{Q_{2\rho,(2\rho)^{2}}(x_{0}, t_{0})}}{\rho^{q}}\iint\limits_{Q_{r,\eta}(x_{0}, t_{0})} (u+\mathcal{I})^{q-2+\frac{\varkappa}{l}}dxdt\right\}.
\end{multline*}
Evidently we have
\begin{multline*}
a^{+}_{Q_{2\rho,(2\rho)^{2}}(x_{0}, t_{0})}\,\,|Q_{r,\eta}(x_{0}, t_{0})|\,\,\mathcal{I}^{q-2+\varkappa\frac{n+p}{ln}}\leqslant\\ \leqslant
\gamma\rho^{q} \mathcal{I}^{\frac{\varkappa p}{ln}} \left\{\frac{1}{\rho^{p}}\iint\limits_{Q_{r,\eta}(x_{0}, t_{0})} (u+\mathcal{I})^{p-2+\frac{\varkappa}{l}}dxdt +\frac{a^{+}_{Q_{2\rho,(2\rho)^{2}}(x_{0}, t_{0})}}{\rho^{q}}\iint\limits_{Q_{r,\eta}(x_{0}, t_{0})} (u+\mathcal{I})^{q-2+\frac{\varkappa}{l}}dxdt\right\},
\end{multline*}
which together with \eqref{eq4.12} yield
\begin{multline}\label{eq4.13}
a^{+}_{Q_{2\rho,(2\rho)^{2}}(x_{0}, t_{0})}\frac{\mathcal{I}^{q-2}}{\rho^{n+q}}\iint\limits_{Q_{r,\eta}(x_{0}, t_{0})} \left(\frac{u}{\mathcal{I}}+1\right)^{q-2+\varkappa\frac{n+p}{ln}}\zeta^{q}(x)dxdt \leqslant \\ \leqslant \gamma \sigma^{-q}\bigg\{\frac{\mathcal{I}^{p-2}}{\rho^{n+p}}\iint\limits_{Q_{r,\eta}(x_{0}, t_{0})} \left(\frac{u}{\mathcal{I}}+1\right)^{p-2+\frac{\varkappa}{l}}dxdt  +\\+ a^{+}_{Q_{2\rho, (2\rho)^{2}}(x_{0},t_{0})}\frac{\mathcal{I}^{q-2}}{\rho^{n+q}}\iint\limits_{Q_{r,\eta}(x_{0}, t_{0})} \left(\frac{u}{\mathcal{I}}+1\right)^{q-2+\frac{\varkappa}{l}} dxdt\bigg\}+\\+ \gamma \rho^{\alpha-q-n}\iint\limits_{Q_{r,\eta}(x_{0}, t_{0})}\left(\frac{u}{\mathcal{I}}\right)^{\varkappa\frac{n+p}{ln}}\,\,u^{q-2}\zeta^{q}(x)dxdt.
\end{multline}
To estimate the last term on the right-hand side of \eqref{eq4.13} we use the Young inequality
\begin{multline}\label{eq4.14}
\rho^{\alpha-q-n}\iint\limits_{Q_{r,\eta}(x_{0}, t_{0})}\left(\frac{u}{\mathcal{I}}\right)^{\varkappa\frac{n+p}{ln}}\,\,u^{q-2}\zeta^{q}(x)dxdt
=\\=\rho^{\alpha-q-n}\iint\limits_{Q_{r,\eta}(x_{0}, t_{0})}\left(\frac{u}{\mathcal{I}}\right)^{\varkappa\frac{n+p}{ln}}\,\,u^{\frac{p-2}{l}+ \frac{(p-2)(l-1)}{l}+q-p}\zeta^{q}(x)dxdt \leqslant \\ \leqslant \frac{\gamma}{\rho^{n+p}}\iint\limits_{Q_{r,\eta}(x_{0}, t_{0})}\left(\frac{u}{\mathcal{I}}+1\right)^{\varkappa\frac{n+p}{n}} (u+\mathcal{I})^{p-2} \zeta^{q}(x) dx\, dt +\\
+\gamma \rho^{(\alpha+p-q)\frac{l}{l-1}-n-p}\iint\limits_{Q_{r,\eta}(x_{0}, t_{0})}\,\,u^{p-2+\frac{(q-p)l}{l-1}} dxdt.
\end{multline}
The first integral on the right-hand side of \eqref{eq4.14} we estimate similarly to \eqref{eq4.11}
\begin{multline}\label{eq4.15}
    \frac{1}{\rho^{n+p}}\iint\limits_{Q_{r,\eta}(x_{0}, t_{0})}\left(\frac{u}{\mathcal{I}}+1\right)^{\varkappa\frac{n+p}{n}} (u+\mathcal{I})^{p-2} \zeta^{q}(x) dx\, dt =\\=\frac{\mathcal{I}^{p-2}}{\rho^{n+p}}\iint\limits_{Q_{r,\eta}(x_{0}, t_{0})}\left(\frac{u}{\mathcal{I}}+1\right)^{p-2+\varkappa\frac{n+p}{n}} \zeta^{q}(x) dx\, dt   \leqslant
\\ \leqslant \gamma \sigma^{-q}\bigg\{\frac{\mathcal{I}^{p-2}}{\rho^{n+p}}\iint\limits_{Q_{r,\eta}(x_{0}, t_{0})} \left(\frac{u}{\mathcal{I}}+1\right)^{p-2+\varkappa}dxdt+
\\+ a^{+}_{Q_{2\rho, (2\rho)^{2}}(x_{0},t_{0})}\frac{\mathcal{I}^{q-2}}{\rho^{n+q}}\iint\limits_{Q_{r,\eta}(x_{0}, t_{0})} \left(\frac{u}{\mathcal{I}}+1\right)^{q-2+\varkappa} dxdt\bigg\}.
\end{multline}
 By our choice of $l$ we estimate the last term on the right-hand side of \eqref{eq4.14} as follows
\begin{multline}\label{eq4.16}
\rho^{(\alpha+p-q)\frac{l}{l-1}-n-p}\iint\limits_{Q_{r,\eta}(x_{0}, t_{0})}u^{p-2+\frac{(q-p)l}{l-1}} dxdt = \\ = \gamma \rho^{(\alpha+p-q)\frac{s-p+2}{q-p}-n-p}\iint\limits_{Q_{r,\eta}(x_{0}, t_{0})}u^{s} dxdt \leqslant \gamma \iint\limits_{Q_{r,\eta}(x_{0}, t_{0})}u^{s} dxdt \leqslant \gamma d^{s}.
\end{multline}
So, collecting estimates \eqref{eq4.11}--\eqref{eq4.16} we arrive at
\begin{multline*}
J_{\sigma}:=\frac{\mathcal{I}^{p-2}}{\rho^{n+p}}\iint\limits_{Q_{(1-\sigma)r,\eta}(x_{0}, t_{0})} \left(\frac{u}{\mathcal{I}}+1\right)^{p-2+\varkappa \frac{n+p}{ln}} dx\, dt +\\+ a^{+}_{Q_{2\rho,(2\rho)^{2}}(x_{0}, t_{0})}\frac{\mathcal{I}^{q-2}}{\rho^{n+q}}\iint\limits_{Q_{(1-\sigma) r,\eta}(x_{0}, t_{0})} \left(\frac{u}{\mathcal{I}}+1\right)^{q-2+\varkappa\frac{n+p}{ln}} dxdt \leqslant\\ \leqslant \gamma \sigma^{-\gamma}
\bigg\{\frac{\mathcal{I}^{p-2}}{\rho^{n+p}}\iint\limits_{Q_{r,\eta}(x_{0}, t_{0})} \left(\frac{u}{\mathcal{I}}+1\right)^{p-2+\varkappa} dx\, dt + \\
+a^{+}_{Q_{2\rho,(2\rho)^{2}}(x_{0}, t_{0})}\frac{\mathcal{I}^{q-2}}{\rho^{n+q}}\iint\limits_{Q_{ r,\eta}(x_{0}, t_{0})} \left(\frac{u}{\mathcal{I}}+1\right)^{q-2+\varkappa}dx\, dt +\\+ \frac{\mathcal{I}^{p-2}}{\rho^{n+p}}\iint\limits_{Q_{r,\eta}(x_{0}, t_{0})} \left(\frac{u}{\mathcal{I}}+1\right)^{p-2+\frac{\varkappa}{l}} dx\, dt + \\
+a^{+}_{Q_{2\rho,(2\rho)^{2}}(x_{0}, t_{0})}\frac{\mathcal{I}^{q-2}}{\rho^{n+q}}\iint\limits_{Q_{ r,\eta}(x_{0}, t_{0})} \left(\frac{u}{\mathcal{I}}+1\right)^{q-2+\frac{\varkappa}{l}}dx\, dt +\gamma\bigg\},
\end{multline*}
which by the Young inequality with any $\epsilon \in (0,1)$  yields
\begin{equation*}
J_{\sigma}\leqslant \epsilon J_{0} +\gamma \sigma^{-\gamma}\epsilon^{-\gamma},
\end{equation*}
from which by iteration we obtain
\begin{equation*}
\frac{\mathcal{I}^{p-2}}{\rho^{n+p}}\iint\limits_{Q_{\frac{15}{8}\rho,\eta}(x_{0}, t_{0})} \left(\frac{u}{\mathcal{I}}+1\right)^{p-2+\varkappa \frac{n+p}{ln}} dx\, dt + \qquad\qquad\qquad\qquad\qquad\qquad\qquad\qquad\qquad\qquad\qquad\qquad\qquad
\end{equation*}
\begin{equation}\label{eq4.17}
 + a^{+}_{Q_{2\rho,(2\rho)^{2}}(x_{0}, t_{0})}\frac{\mathcal{I}^{q-2}}{\rho^{n+q}}\iint\limits_{Q_{\frac{15}{8}\rho,\eta}(x_{0}, t_{0})} \left(\frac{u}{\mathcal{I}}+1\right)^{q-2+\varkappa\frac{n+p}{ln}} dxdt \leqslant \gamma.
\end{equation}
To complete the proof of the lemma we need to obtain the reverse H\"{o}lder inequality. Define the number
$\bar{\varkappa} \leqslant \varkappa$ by the condition
\begin{equation*}
(m-q+2) \bigg(\frac{ln}{n+p}\bigg)^{j+1} =\bar{\varkappa},
\end{equation*}
in this setting
\begin{multline*}
\bar{\varkappa}\bigg(\frac{n+p}{ln}\bigg)^{i}=(m -q+2)\bigg(\frac{ln}{n+p}\bigg)^{j+1-i} <\\<
(1+\frac{p-n(l-1)}{ln})\bigg(\frac{ln}{n+p}\bigg)^{j+1-i}=\bigg(\frac{ln}{n+p}\bigg)^{j-i}\leqslant 1,\,\,1\leqslant i\leqslant j.
\end{multline*}
We use Lemma \ref{lem2.4} with $\varepsilon= \dfrac{\bar{\varkappa}}{l}\bigg(\dfrac{n+p}{ln}\bigg)^{i} $ for the pair of cylinders
\\$Q_{i}:=B_{i}\times(t_{0}, t_{0}+\eta_{i})$ and $Q_{i+1}$, $B_{i}:=B_{\rho_{i}}(x_{0})$, $\eta_{i}=\dfrac{\mathcal{I}^{2}}{\varphi^{+}_{Q_{2\rho, (2\rho)^{2}}(x_{0}, t_{0})}\big(\frac{\mathcal{I}}{\rho_{i}}\big)}$, \\$\rho_{i}:=\dfrac{15}{8} \rho(1-\dfrac{1}{15}\dfrac{1-2^{-i+1}}{1-2^{-(j+1)}})$, $i=1,...,j$. Choose $\zeta_{1}(x) \in C^{1}_{0}(B_{i})$,
$\zeta_{1}(x)=1$ in $B_{i+1}$, $0\leqslant \zeta_{1}(x) \leqslant 1$, $|\nabla \zeta_{1}(x)| \leqslant \gamma \dfrac{2^{i}}{\rho}$ and $\zeta_{2}(t)\in C^{1}(\mathbb{R}_{+})$, $0\leqslant \zeta_{2}(t)\leqslant 1$, $\zeta_{2}(t)=1$ for $t\leqslant t_{0}+ \eta_{i+1}$, $\zeta_{2}=0$ for $t \geqslant t_{0}+\eta_{i}$, $\bigg| \dfrac{d}{dt}\zeta_{2}(t)\bigg| \leqslant \gamma 2^{i\gamma}\dfrac{\varphi^{+}_{Q_{2\rho, (2\rho)^{2}}(x_{0}, t_{0})}\big(\frac{\mathcal{I}}{\rho}\big)}{\mathcal{I}^{2}}$.

By the Sobolev embedding theorem, choosing $y_{i}:=p-2 +\bar{\varkappa}\bigg(\dfrac{n+p}{ln}\bigg)^{i}$, \\$z_{i}:=q-2 +\bar{\varkappa}\bigg(\dfrac{n+p}{ln}\bigg)^{i}$ we obtain
\begin{multline*}
\iint\limits_{Q_{i+1}} (u+\mathcal{I})^{y_{i+1}} dxdt \leqslant \gamma \bigg(\sup\limits_{t_{0}<t<t_{0}+\eta_{i}}\int\limits_{B_{i}}(u+\mathcal{I})^{\frac{\bar{\varkappa}}{l}(\frac{n+p}{ln})^{i}}(\zeta_{1}(x)\zeta_{2}(t))^{q} dx\bigg)^{\frac{p}{n}} \times\\\times\iint\limits_{Q_{i}}(u+\mathcal{I})^{-2+\frac{\bar{\varkappa}}{l}(\frac{n+p}{ln})^{i}}|\big(\nabla u (\zeta_{1}(x) \zeta_{2}(t))^{\frac{q}{p}}\big)|^{p} dx\, dt
\leqslant\\ \leqslant \gamma 2^{\gamma i} \bigg\{\frac{\varphi^{+}_{Q_{2\rho, (2\rho)^{2}}(x_{0}, t_{0})}\big(\frac{\mathcal{I}}{\rho}\big)}{\mathcal{I}^{2}}\iint\limits_{Q_{i}}(u+\mathcal{I})^{\frac{\bar{\varkappa}}{l}(\frac{n+p}{ln})^{i}} dx\, dt +\rho^{-p}\iint\limits_{Q_{i}}(u+\mathcal{I})^{p-2+\frac{\bar{\varkappa}}{l}(\frac{n+p}{ln})^{i}} dx\, dt +\\+ \frac{a^{+}_{Q_{2\rho,(2\rho)^{2}}(x_{0},t_{0})}}{\rho^{q}}\iint\limits_{Q_{i}}(u+\mathcal{I})^{q-2+\frac{\bar{\varkappa}}{l}(\frac{n+p}{ln})^{i}} dx\, dt \bigg\}^{1+\frac{p}{n}}\leqslant\\ \leqslant \gamma 2^{\gamma i}\bigg\{\rho^{-p}\iint\limits_{Q_{i}}(u+\mathcal{I})^{p-2+\frac{\bar{\varkappa}}{l}(\frac{n+p}{ln})^{i}} dx\, dt + \frac{a^{+}_{Q_{2\rho,(2\rho)^{2}}(x_{0},t_{0})}}{\rho^{q}}\iint\limits_{Q_{i}}(u+\mathcal{I})^{q-2+\frac{\bar{\varkappa}}{l}(\frac{n+p}{ln})^{i}} dx\, dt \bigg\}^{1+\frac{p}{n}},
\end{multline*}
which by the H\"{o}lder inequality and  \eqref{eq4.17} yields
\begin{multline}\label{eq4.18}
\frac{\mathcal{I}^{p-2}}{\rho^{n+p}}\iint\limits_{{Q_{i+1}}} \left(\frac{u}{\mathcal{I}}+1\right)^{y_{i+1}} dxdt \leqslant
\gamma 2^{\gamma i}\bigg\{\frac{\mathcal{I}^{p-2}}{\rho^{n+p}}\iint\limits_{{Q_{i}}} \left(\frac{u}{\mathcal{I}}+1\right)^{p-2+\frac{\bar{\varkappa}}{l}(\frac{n+p}{ln})^{i}} dxdt +\\+ a^{+}_{Q_{2\rho,(2\rho)^{2}}(x_{0},t_{0})}\frac{\mathcal{I}^{q-2}}{\rho^{n+q}}\iint\limits_{Q_{i}}\left(\frac{u}{\mathcal{I}}+1\right)^{q-2+\frac{\bar{\varkappa}}{l}(\frac{n+p}{ln})^{i}} dx\, dt\bigg\}^{1+\frac{p}{n}} \leqslant \\ \leqslant \gamma 2^{\gamma i}\bigg\{\frac{\mathcal{I}^{p-2}}{\rho^{n+p}}\bigg(\iint\limits_{{Q_{i}}} \left(\frac{u}{\mathcal{I}}+1\right)^{p-2+\bar{\varkappa}(\frac{n+p}{ln})^{i}} dxdt\bigg)^{\frac{1}{l}}\bigg(\iint\limits_{{Q_{i}}} \left(\frac{u}{\mathcal{I}}+1\right)^{p-2} dxdt\bigg)^{1-\frac{1}{l}} +\\+ a^{+}_{Q_{2\rho,(2\rho)^{2}}(x_{0},t_{0})}\frac{\mathcal{I}^{q-2}}{\rho^{n+q}}\bigg(\iint\limits_{Q_{i}}\left(\frac{u}{\mathcal{I}}+1\right)^{q-2+\bar\varkappa(\frac{n+p}{ln})^{i}} dx\, dt\bigg)^{\frac{1}{l}}\bigg(\iint\limits_{Q_{i}}\left(\frac{u}{\mathcal{I}}+1\right)^{q-2} dx\, dt\bigg)^{1-\frac{1}{l}}\bigg\}^{1+\frac{p}{n}} \leqslant \\
\leqslant \gamma 2^{\gamma i}\bigg\{\frac{\mathcal{I}^{p-2}}{\rho^{n+p}}\iint\limits_{{Q_{i}}} \negthickspace \left(\frac{u}{\mathcal{I}}+1\right)^{y_{i}} \negthickspace dxdt+a^{+}_{Q_{2\rho,(2\rho)^{2}}(x_{0},t_{0})}\frac{\mathcal{I}^{q-2}}{\rho^{n+q}}\iint\limits_{Q_{i}}\negthickspace\left(\frac{u}{\mathcal{I}}+1\right)^{z_{i}} \negthickspace dx\, dt\bigg\}^{\frac{n+p}{ln}} .
\end{multline}
Similarly,
\begin{multline*}
a^{-}_{Q_{2\rho, (2\rho)^{2}}(x_{0},t_{0})}\iint\limits_{Q_{i+1}} (u+\mathcal{I})^{y_{i+1}} dxdt \leqslant \\ \leqslant \gamma \rho^{q-p}\bigg(\sup\limits_{t_{0}<t<t_{0}+\eta_{i}}\int\limits_{B_{i}}(u+\mathcal{I})^{\frac{\bar{\varkappa}}{l}(\frac{n+p}{ln})^{i}}(\zeta_{1}(x)\zeta_{2}(t))^{q} dx\bigg)^{\frac{p}{n}} \times \\ \times \iint\limits_{Q_{i}}a(x,t)(u+\mathcal{I})^{-2+\frac{\bar{\varkappa}}{l}(\frac{n+p}{ln})^{i}} |\big(\nabla u \zeta_{1}(x) \zeta_{2}(t)\big)|^{q} dx\, dt \leqslant \\
\leqslant \gamma 2^{\gamma i} \bigg\{\frac{\varphi^{+}_{Q_{2\rho, (2\rho)^{2}}(x_{0}, t_{0})}\big(\frac{\mathcal{I}}{\rho}\big)}{\mathcal{I}^{2}}\times\\
\times\iint\limits_{Q_{i}}(u+\mathcal{I})^{\frac{\bar{\varkappa}}{l}(\frac{n+p}{ln})^{i}} dx\, dt + \rho^{-p}\iint\limits_{Q_{i}}(u+\mathcal{I})^{p-2+\frac{\bar{\varkappa}}{l}(\frac{n+p}{ln})^{i}} dx\, dt + \\
+\frac{a^{+}_{Q_{2\rho,(2\rho)^{2}}(x_{0},t_{0})}}{\rho^{q}}\iint\limits_{Q_{i}}(u+\mathcal{I})^{q-2+\frac{\bar{\varkappa}}{l}(\frac{n+p}{ln})^{i}} dx\, dt \bigg\}^{1+\frac{p}{n}}
\leqslant \\ \leqslant \gamma 2^{\gamma i}\bigg\{\rho^{-p}\iint\limits_{Q_{i}}(u+\mathcal{I})^{p-2+\frac{\bar{\varkappa}}{l}(\frac{n+p}{ln})^{i}} dx\, dt + \\
+\frac{a^{+}_{Q_{2\rho,(2\rho)^{2}}(x_{0},t_{0})}}{\rho^{q}}\iint\limits_{Q_{i}}(u+\mathcal{I})^{q-2+\frac{\bar{\varkappa}}{l}(\frac{n+p}{ln})^{i}} dx\, dt \bigg\}^{1+\frac{p}{n}},
\end{multline*}
which by the H\"{o}lder inequality and  \eqref{eq4.17} yields
\begin{multline}\label{eq4.19}
a^{-}_{Q_{2\rho,(2\rho)^{2}}(x_{0},t_{0})}\frac{\mathcal{I}^{q-2}}{\rho^{n+q}}\iint\limits_{Q_{i+1}}\left(\frac{u}{\mathcal{I}}+1\right)^{z_{i+1}} dx\, dt\leqslant\\ \leqslant \gamma 2^{\gamma i}\bigg\{\frac{\mathcal{I}^{p-2}}{\rho^{n+p}}\iint\limits_{{Q_{i}}} \left(\frac{u}{\mathcal{I}}+1\right)^{y_{i}} dxdt+a^{+}_{Q_{2\rho,(2\rho)^{2}}(x_{0},t_{0})}\frac{\mathcal{I}^{q-2}}{\rho^{n+q}}\iint\limits_{Q_{i}}\left(\frac{u}{\mathcal{I}}+1\right)^{z_{i}} dx\, dt\bigg\}^{\frac{n+p}{ln}} .
\end{multline}
Furthermore, by condition ($A$)
\begin{multline}\label{eq4.20}
a^{+}_{Q_{2\rho,(2\rho)^{2}}(x_{0},t_{0})}\frac{\mathcal{I}^{q-2}}{\rho^{n+q}}\iint\limits_{Q_{i+1}}\Big(\frac{u}{\mathcal{I}}+1\Big)^{z_{i+1}} dx\, dt\leqslant \gamma a^{+}_{Q_{2\rho,(2\rho)^{2}}(x_{0},t_{0})}\frac{\mathcal{I}^{q-2}}{\rho^{n+q}} |Q_{i+1}|+\\ + \gamma a^{-}_{Q_{2\rho,(2\rho)^{2}}(x_{0},t_{0})}\frac{\mathcal{I}^{q-2}}{\rho^{n+q}}\iint\limits_{Q_{i+1}}\Big(\frac{u}{\mathcal{I}}\Big)^{z_{i+1}} dx\, dt +\rho^{\alpha-n-q}\iint\limits_{Q_{i+1}}\left(\frac{u}{\mathcal{I}}\right)^{\bar{\varkappa}(\frac{n+p}{ln})^{i+1}}u^{q-2} dxdt. 
\end{multline}
The first term on the right-hand side of \eqref{eq4.20} we estimate using the Young inequality
\begin{multline}\label{eq4.21}
 a^{+}_{Q_{2\rho,(2\rho)^{2}}(x_{0},t_{0})}\frac{\mathcal{I}^{q-2}}{\rho^{n+q}} |Q_{i+1}|\leqslant 
\gamma  \frac{a^{+}_{Q_{2\rho,(2\rho)^{2}}(x_{0},t_{0})}}{\rho^{n+q}}
\iint\limits_{Q_{i}}\left(\frac{u}{\mathcal{I}}+1\right)^{\bar{\varkappa}(\frac{n+p}{ln})^{i}}(u+\mathcal{I})^{q-2} dx\, dt=\\=\gamma a^{+}_{Q_{2\rho,(2\rho)^{2}}(x_{0},t_{0})}\frac{\mathcal{I}^{q-2}}{\rho^{n+q}}\iint\limits_{Q_{i}}\left(\frac{u}{\mathcal{I}}+1\right)^{z_{i}} dx\, dt
\leqslant \\ \leqslant \gamma \Big\{1+a^{+}_{Q_{2\rho,(2\rho)^{2}}(x_{0},t_{0})}\frac{\mathcal{I}^{q-2}}{\rho^{n+q}}\iint\limits_{Q_{i}}\left(\frac{u}{\mathcal{I}}+1\right)^{z_{i}} dx\, dt\Big\}^{\frac{n+p}{ln}}.
\end{multline}
To estimate the last term on the right-hand side of \eqref{eq4.20} we use the H\"{o}lder inequality, by our choice of $l$
\begin{multline*}
\rho^{\alpha-n-q}\iint\limits_{Q_{i+1}}\left(\frac{u}{\mathcal{I}}\right)^{\bar{\varkappa}(\frac{n+p}{ln})^{i+1}}u^{q-2} dxdt=\\=\rho^{\alpha-n-q}\iint\limits_{Q_{i+1}}\left(\frac{u}{\mathcal{I}}\right)^{\bar{\varkappa}(\frac{n+p}{ln})^{i+1}}\,u^{\frac{p-2}{l}+\frac{(p-2)(l-1)}{l}+q-p} dxdt\leqslant \\ \leqslant \gamma \rho^{\alpha-n-q}\bigg(\iint\limits_{Q_{i+1}}\left(\frac{u}{\mathcal{I}}+1\right)^{\bar{\varkappa}\frac{n+p}{n}(\frac{n+p}{ln})^{i}}(u+\mathcal{I})^{p-2} dxdt\bigg)^{\frac{1}{l}} \bigg(\iint\limits_{Q_{i+1}}u^{s} dx\, dt\bigg)^{\frac{l-1}{l}}\leqslant \\ \leqslant \gamma \rho^{\alpha+p-q-\frac{(n+p)(q-p)}{s-p+2}}\bigg(\frac{\mathcal{I}^{p-2}}{\rho^{n+p}}\iint\limits_{Q_{i+1}}\left(\frac{u}{\mathcal{I}}+1\right)^{p-2+\bar{\varkappa}\frac{n+p}{n}(\frac{n+p}{ln})^{i}} dxdt\bigg)^{\frac{1}{l}}\leqslant\\\leqslant 
\gamma \bigg(\frac{\mathcal{I}^{p-2}}{\rho^{n+p}}\iint\limits_{Q_{i+1}}\left(\frac{u}{\mathcal{I}}+1\right)^{p-2+\bar{\varkappa}\frac{n+p}{n}(\frac{n+p}{ln})^{i}} dxdt\bigg)^{\frac{1}{l}}.
\end{multline*}
The integral on the right-hand side of this inequality we estimate similarly to \eqref{eq4.18}
\begin{multline*}
\bigg(\frac{\mathcal{I}^{p-2}}{\rho^{n+p}}\iint\limits_{Q_{i+1}}\left(\frac{u}{\mathcal{I}}+1\right)^{p-2+\bar{\varkappa}\frac{n+p}{n}(\frac{n+p}{ln})^{i}} dxdt\bigg)^{\frac{1}{l}} \leqslant\\ \leqslant \gamma 2^{\gamma i}\bigg\{\frac{\mathcal{I}^{p-2}}{\rho^{n+p}}\iint\limits_{{Q_{i}}} \left(\frac{u}{\mathcal{I}}+1\right)^{y_{i}} dxdt+a^{+}_{Q_{2\rho,(2\rho)^{2}}(x_{0},t_{0})}\frac{\mathcal{I}^{q-2}}{\rho^{n+q}}\iint\limits_{Q_{i}}\left(\frac{u}{\mathcal{I}}+1\right)^{z_{i}} dx \,dt\bigg\}^{\frac{n+p}{ln}} .
\end{multline*}
Therefore, inequalities \eqref{eq4.19}--\eqref{eq4.21} yield
\begin{multline}\label{eq4.22}
a^{+}_{Q_{2\rho,(2\rho)^{2}}(x_{0},t_{0})}\frac{\mathcal{I}^{q-2}}{\rho^{n+q}}\iint\limits_{Q_{i+1}}\Big(\frac{u}{\mathcal{I}}+1\Big)^{z_{i+1}} dx\, dt\leqslant\\ \leqslant\gamma 2^{\gamma i}\bigg\{1+\frac{\mathcal{I}^{p-2}}{\rho^{n+p}}\iint\limits_{{Q_{i}}} \left(\frac{u}{\mathcal{I}}+1\right)^{y_{i}} dxdt+a^{+}_{Q_{2\rho,(2\rho)^{2}}(x_{0},t_{0})}\frac{\mathcal{I}^{q-2}}{\rho^{n+q}}\iint\limits_{Q_{i}}\left(\frac{u}{\mathcal{I}}+1\right)^{z_{i}} dx \,dt\bigg\}^{\frac{n+p}{ln}} .
\end{multline}
Collecting estimates \eqref{eq4.18}, \eqref{eq4.22} we arrive at
\begin{multline*}
J_{i+1}:=\bigg(\frac{\mathcal{I}^{p-2}}{\rho^{n+p}}\iint\limits_{{Q_{i+1}}} \left(\frac{u}{\mathcal{I}}+1\right)^{y_{i+1}} dxdt +a^{+}_{Q_{2\rho,(2\rho)^{2}}(x_{0},t_{0})}\frac{\mathcal{I}^{q-2}}{\rho^{n+q}}\iint\limits_{Q_{i+1}}\left(\frac{u}{\mathcal{I}}+1\right)^{z_{i+1}} dx\, dt\bigg)^
{(\frac{l n}{n+p})^{i+1}} \leqslant \\ \leqslant  \gamma 2^{\gamma i(\frac{l n}{n+p})^{i+1}} \big(1+J_{i}\big),\quad i=1,2,...,j.
\end{multline*}

From this  and \eqref{eq4.17}, after a finite number of iterations, we obtain
\begin{equation*}
J_{j+1}\leqslant \gamma(j) \big(1+ J_{1}\big)\leqslant \gamma,
\end{equation*}
 which proves \eqref{eq4.10}. This  completes the proof of the lemma.
\end{proof}

Next is the gradient estimate for the super-solution.
\begin{lemma}\label{lem4.3}
For all $\delta \in (0,\dfrac{5}{8})$ there holds
\begin{equation}\label{eq4.23}
\frac{1}{\rho}\iint\limits_{Q_{\frac{13}{8}\rho,\delta \eta}(x_{0}, t_{0})} |\nabla u|^{p-1} dx\, dt +
\frac{1}{\rho}\iint\limits_{Q_{\frac{13}{8}\rho,\delta \eta}(x_{0}, t_{0})} a(x,t) |\nabla u|^{q-1} dx\, dt \leqslant \gamma \delta^{\frac{\varepsilon}{p(1+2\varepsilon)}} \mathcal{I}\rho^{n},
\end{equation}
where $\eta=\dfrac{\mathcal{I}^{2}}{\varphi^{+}_{Q_{2\rho, (2\rho)^{2}}(x_{0},t_{0})}\big(\frac{\mathcal{I}}{\rho}\big)}$, $\varepsilon=\dfrac{lp-n(l-1)}{4(q-1)ln}$ and $l=\dfrac{s-p+2}{s-q+2}$.
\end{lemma}
\begin{proof}
By the H\"{o}lder inequality  we have
\begin{multline*}
\frac{1}{\rho}\iint\limits_{Q_{\frac{13}{8}\rho,\delta \eta}(x_{0}, t_{0})}| \nabla u|^{p-1} dxdt +  \frac{1}{\rho}\iint\limits_{Q_{\frac{13}{8}\rho,\delta \eta}(x_{0}, t_{0})} a(x,t) |\nabla u|^{q-1} dx\, dt \leqslant \\ \leqslant \gamma \frac{1}{\rho}\bigg(\iint\limits_{Q_{\frac{13}{8}\rho,\frac{5}{8} \eta}(x_{0}, t_{0})}\left(\frac{u}{\mathcal{I}}+1\right)^{-1-\varepsilon}| \nabla u|^{p} dxdt\bigg)^{\frac{p-1}{p}}\times\\\times\bigg(\iint\limits_{Q_{\frac{13}{8}\rho,\delta \eta}(x_{0}, t_{0})}\left(\frac{u}{\mathcal{I}}+1\right)^{(1+\varepsilon)(p-1)}dx\, dt\bigg)^{\frac{1}{p}}+\\+ \gamma\frac{1}{\rho} \bigg(\iint\limits_{Q_{\frac{13}{8}\rho,\frac{5}{8} \eta}(x_{0}, t_{0})}\negthickspace\negthickspace a(x,t)\left(\frac{u}{\mathcal{I}}+1\right)^{-1-\varepsilon}| \nabla u|^{q} dxdt\bigg)^{\frac{q-1}{q}}\times\\\times  \bigg(\iint\limits_{Q_{\frac{13}{8}\rho,\delta \eta}(x_{0}, t_{0})}\negthickspace\negthickspace a(x,t)\left(\frac{u}{\mathcal{I}}+1\right)^{(1+\varepsilon)(q-1)}dx\, dt\bigg)^{\frac{1}{q}}.
\end{multline*}
By Lemma \ref{lem4.2}
\begin{multline*}
\iint\limits_{Q_{\frac{13}{8}\rho,\delta \eta}(x_{0}, t_{0})}\left(\frac{u}{\mathcal{I}}+1\right)^{(1+\varepsilon)(p-1)}dx\, dt \leqslant\\\leqslant\gamma
\bigg(\iint\limits_{Q_{\frac{13}{8}\rho,\frac{5}{8} \eta}(x_{0}, t_{0})}\negthickspace\negthickspace\negthickspace\left(\frac{u}{\mathcal{I}}+1\right)^{(1+2\varepsilon)(p-1)}dx\, dt\bigg)^{\frac{1+\varepsilon}{1+2\varepsilon}}(\delta \rho^{n}\eta)^{\frac{\varepsilon}{1+2\varepsilon}} \leqslant
\gamma \delta^{\frac{\varepsilon}{1+2\varepsilon}} \rho^{n+p}\mathcal{I}^{2-p},
\end{multline*}
and
\begin{multline*}
\iint\limits_{Q_{\frac{13}{8}\rho,\delta \eta}(x_{0}, t_{0})}a(x,t)\left(\frac{u}{\mathcal{I}}+1\right)^{(1+\varepsilon)(q-1)}dx\, dt
\leqslant\\\leqslant a^{+}_{Q_{2\rho,(2\rho)^{2}}(x_{0},t_{0})}\iint\limits_{Q_{\frac{13}{8}\rho,\delta \eta}(x_{0}, t_{0})}\left(\frac{u}{\mathcal{I}}+1\right)^{(1+\varepsilon)(q-1)}dx\, dt \leqslant \gamma \delta^{\frac{\varepsilon}{1+2\varepsilon}} \rho^{n+q} \mathcal{I}^{2-q}.
\end{multline*}
By Lemma \ref{lem2.4} with the appropriate choice of $\zeta_{1}(x)$, $\zeta_{2}(t)$, $| \nabla \zeta_{1}(x)| \leqslant \dfrac{8}{\rho}$, $|\dfrac{d}{dt} \zeta_{2}(t)| \leqslant \dfrac{8}{\eta}$ we obtain
\begin{multline*}
\iint\limits_{Q_{\frac{13}{8}\rho,\frac{5}{8} \eta}(x_{0}, t_{0})}\left(\frac{u}{\mathcal{I}}+1\right)^{-1-\varepsilon}| \nabla u|^{p} dxdt+
\iint\limits_{Q_{\frac{13}{8}\rho,\frac{5}{8} \eta}(x_{0}, t_{0})}a(x,t)\left(\frac{u}{\mathcal{I}}+1\right)^{-1-\varepsilon}| \nabla u|^{q} dxdt\leqslant\\
\leqslant \gamma \bigg(\frac{\mathcal{I}^{2}}{\eta}\iint\limits_{Q_{\frac{7}{4}\rho,\frac{3}{4} \eta}(x_{0}, t_{0})}\left(\frac{u}{\mathcal{I}}+1\right)^{1-\varepsilon}dxdt 
+\frac{\mathcal{I}^{p}}{\rho^{p}}\iint\limits_{Q_{\frac{7}{4}\rho,\frac{3}{4} \eta}(x_{0}, t_{0})}\left(\frac{u}{\mathcal{I}}+1\right)^{p-1-\varepsilon}dxdt+\\
+a^{+}_{Q_{2\rho,(2\rho)^{2}}(x_{0},t_{0})}
\frac{\mathcal{I}^{q}}{\rho^{q}}\iint\limits_{Q_{\frac{7}{4}\rho,\frac{3}{4} \eta}(x_{0}, t_{0})}\left(\frac{u}{\mathcal{I}}+1\right)^{q-1-\varepsilon}dxdt \bigg) \leqslant \\ \leqslant \gamma\bigg(\frac{\mathcal{I}^{p}}{\rho^{p}}\iint\limits_{Q_{\frac{7}{4}\rho,\frac{3}{4} \eta}(x_{0}, t_{0})}\left(\frac{u}{\mathcal{I}}+1\right)^{p-1-\varepsilon}dxdt+ \\
+a^{+}_{Q_{2\rho,(2\rho)^{2}}(x_{0},t_{0})}\frac{\mathcal{I}^{q}}{\rho^{q}}\iint\limits_{Q_{\frac{7}{4}\rho,\frac{3}{4} \eta}(x_{0}, t_{0})}\left(\frac{u}{\mathcal{I}}+1\right)^{q-1-\varepsilon}dxdt \bigg),
\end{multline*}
which by Lemma \ref{lem4.2} yields
\begin{multline*}
\iint\limits_{Q_{\frac{13}{8}\rho,\frac{5}{8} \eta}(x_{0}, t_{0})}\left(\frac{u}{\mathcal{I}}+1\right)^{-1-\varepsilon}| \nabla u|^{p} dxdt+\\+
\iint\limits_{Q_{\frac{13}{8}\rho,\frac{5}{8} \eta}(x_{0}, t_{0})}a(x,t)\left(\frac{u}{\mathcal{I}}+1\right)^{-1-\varepsilon}| \nabla u|^{q} dxdt\leqslant \gamma \mathcal{I}^{2} \rho^{n}.
\end{multline*}
Collecting the previous inequalities we arrive at the required \eqref{eq4.23}, which completes the proof of the lemma.
\end{proof}

The following lemma gives us the possibility to prove inequality \eqref{eq3.26} from Theorem \ref{th3.1}.
\begin{lemma}\label{lem4.4}
There exists $\delta_{0}\in (0,\dfrac{5}{8})$ such that
\begin{equation}\label{eq4.24}
\inf\limits_{t_{0}< t< t_{0}+\delta_{0}\eta}\fint\limits_{B_{\frac{3}{2}\rho}(x_{0})} u(x,t) dx \geqslant \frac{\mathcal{I}}{2^{n+1}},
\end{equation}
where $\eta=\dfrac{\mathcal{I}^{2}}{\varphi^{+}_{Q_{2\rho,(2\rho)^{2}}(x_{0},t_{0})}(\frac{\mathcal{I}}{\rho})}$.
\end{lemma}
\begin{proof}
Let $\zeta(x) \in C^{1}_{0}(B_{\frac{3}{2}\rho}(x_{0}))$, $\zeta(x)=1$ in $B_{\rho}(x_{0})$, $0 \leqslant \zeta(x) \leqslant 1$,
$|\nabla \zeta(x)| \leqslant \dfrac{2}{\rho}$, test \eqref{eq1.4} by $ \zeta^{q}(x)$, then by Lemma \ref{lem4.3} we obtain for any $\delta \in \big(0,\frac{5}{8}\big)$
\begin{multline*}
\int\limits_{B_{\rho}(x_{0})} u(x, t_{0}) dx \leqslant \int\limits_{B_{\frac{3}{2}\rho}(x_{0})} u(x,t) dx +\frac{\gamma}{\rho}
\iint\limits_{Q_{\frac{3}{2}\rho,\delta \eta}(x_{0},t_{0})} |\nabla u|^{p-1} dxdt +\\+\frac{\gamma}{\rho}
\iint\limits_{Q_{\frac{3}{2}\rho,\delta \eta}(x_{0},t_{0})} a(x,t)|\nabla u|^{q-1} dxdt \leqslant \\ \leqslant   \int\limits_{B_{\frac{3}{2}\rho}(x_{0})} u(x,t) dx +\gamma \delta^{\frac{\varepsilon}{p(1+2\varepsilon)}}\mathcal{I} \rho^{n},\quad t_{0}<t<t_{0}+\delta\eta,
\end{multline*}
from which the required \eqref{eq4.24} follows, provided that $\delta$ is small enough.
\end{proof}

To complete the proof of Theorem \ref{th1.1} we note that by the H\"{o}lder inequality and Lemma \ref{lem4.3}
\begin{multline*}
\underset{Q_{\frac{3}{2}\rho, \delta_{0}\eta}}{\fint\fint}\negthickspace\negthickspace\negthickspace\left(\frac{u}{\mathcal{I}}+1\right)^{m}\negthickspace dxdt \leqslant
\gamma \frac{\mathcal{I}^{p-2}}{\delta_{0} \rho^{n+p}}\iint\limits_{Q_{\frac{3}{2}\rho, \delta_{0}\eta}}\negthickspace\negthickspace\negthickspace\left(\frac{u}{\mathcal{I}}+1\right)^{m}\negthickspace dxdt+\\+\gamma a^{+}_{Q_{2\rho, (2\rho)^{2}}(x_{0},t_{0})} \frac{\mathcal{I}^{q-2}}{\delta_{0} \rho^{n+q}}\iint\limits_{Q_{\frac{3}{2}\rho, \delta_{0}\eta}}\negthickspace\negthickspace\negthickspace\left(\frac{u}{\mathcal{I}}+1\right)^{m} \negthickspace dxdt \leqslant \frac{\gamma}{\delta_{0}}
\end{multline*}
with some $m >1$. This inequality implies the existence of $\bar{t} \in (t_{0}, t_{0} +\delta_{0} \eta)$ such that
\begin{equation}\label{eq4.25}
\fint\limits_{B_{\frac{3}{2}\rho}(x_{0})} u^{m}(x, \bar{t}) dx \leqslant \frac{\gamma}{\delta_{0}} \mathcal{I}^{m}.
\end{equation}
Therefore
\begin{multline*}
\frac{\mathcal{I}}{2^{n+1}} \leqslant \fint\limits_{B_{\frac{3}{2}\rho}(x_{0})}u(x,\bar{t}) dx =\\=\fint\limits_{B_{\frac{3}{2}\rho}(x_{0})
\cap\{u(\cdot,\bar{t})\leqslant \frac{\mathcal{I}}{2^{n+2}}\}}u(x,\bar{t}) dx+\fint\limits_{B_{\frac{3}{2}\rho}(x_{0})\cap\{u(\cdot,\bar{t})\geqslant \frac{\mathcal{I}}{2^{n+2}}\}}u(x,\bar{t}) dx \leqslant \\ \leqslant
\frac{\mathcal{I}}{2^{n+2}} + \Big(\fint\limits_{B_{\frac{3}{2}\rho}(x_{0})} u^{m}(x, \bar{t}) dx\Big)^{\frac{1}{m}}
\left(\frac{|\big\{B_{\frac{3}{2}\rho}(x_{0}): u(\cdot, \bar{t}) \geqslant\frac{\mathcal{I}}{2^{n+2}}\big\}|}{|B_{\frac{3}{2}\rho}(x_{0})|}\right)^{1-\frac{1}{m}},
\end{multline*}
which together with \eqref{eq4.25} yields
\begin{equation}\label{eq4.26}
\left|\left\{B_{\frac{3}{2}\rho}(x_{0}): u(\cdot, \bar{t}) >\frac{\mathcal{I}}{2^{n+2}}\right\}\right| \geqslant \gamma(\delta_{0})|B_{\frac{3}{2}\rho}(x_{0})|,\quad \gamma(\delta_{0}) \in (0,1).
\end{equation}
Inequality \eqref{eq3.25}  of Theorem \ref{th3.1} similarly to Lemma \ref{lem4.1}  follows from \eqref{eq4.26}, so, Theorem \ref{th3.1} is applicable. Using Theorem \ref{th3.1} with $\beta=\gamma(\delta_{0})$, from \eqref{eq4.26} we arrive at
\begin{equation}\label{eq4.27}
u(x,t)\geqslant \frac{ \gamma(\delta_{0})^{B}\sigma_{1}}{2^{n+2}}\,\mathcal{I}=\bar{\sigma}^{*}_{1}\,\mathcal{I},\quad x\in B_{4\rho}(x_{0}),
\end{equation}
provided that $C_{1}\geqslant \dfrac{2^{n+2} C}{\gamma(\delta_{0})^{B}}$ and for  all time levels
$$t_{0}+\bar{B}^{*}_{1}\dfrac{(\bar{\sigma}^{*}_{1}\mathcal{I})^{2}}{\varphi^{+}_{Q_{14\rho,(14\rho)^{2}}(x_{0},t_{0})}( \frac{\bar{\sigma}^{*}_{1}\mathcal{I}}{\rho})} \leqslant t \leqslant t_{0} +\bar{B}^{*}_{2}\dfrac{(\bar{\sigma}^{*}_{1}\mathcal{I})^{2}}{\varphi^{+}_{Q_{14\rho,(14\rho)^{2}}(x_{0},t_{0})}(\frac{\bar{\sigma}^{*}_{1}\mathcal{I}}{\rho})}.$$

A further application of the expansion of positivity Theorem \ref{th3.1} implies from \eqref{eq4.5} and \eqref{eq4.27} the bound below \eqref{eq1.7}, with new constant depending only on the data,  for all times $t$ as in \eqref{eq1.8}.
This completes the proof of Theorem \ref{th1.1}.

\vskip3.5mm
{\bf Acknowledgements.} The authors are partially supported by a grant of the National Academy of Sciences of Ukraine (project number is 0120U100178), and by the European Union as part of a MSCA4Ukraine  Postdoctoral Fellowship.

\newpage
CONTACT INFORMATION

\medskip
\textbf{Mariia O.~Savchenko}  \\Technische Universi{\"a}t Braunschweig,\\
\indent Universit{\"a}tsplatz 2, 38106 Braunschweig, Germany
\\Institute of Applied Mathematics and Mechanics,
National Academy of Sciences of Ukraine, \\ \indent Heneral Batiuk Str. 19, 84116 Sloviansk, Ukraine\\
shan\underline{ }maria@ukr.net

\medskip
\textbf{Igor I.~Skrypnik}\\Institute of Applied Mathematics and Mechanics,
National Academy of Sciences of Ukraine, \\ \indent Heneral Batiuk Str. 19, 84116 Sloviansk, Ukraine\\
Vasyl' Stus Donetsk National University,
\\ \indent 600-richcha Str. 21, 21021 Vinnytsia, Ukraine\\ihor.skrypnik@gmail.com

\medskip
\textbf{Yevgeniia A. Yevgenieva}
\\ Max Planck Institute for Dynamics of Complex Technical Systems, \\ \indent Sandtorstrasse 1, 39106 Magdeburg, Germany
\\Institute of Applied Mathematics and Mechanics,
National Academy of Sciences of Ukraine, \\ \indent Heneral Batiuk Str. 19, 84116 Sloviansk, Ukraine\\yevgeniia.yevgenieva@gmail.com

\end{document}